\documentclass[11pt,twoside]{article}

\usepackage{amssymb}
\usepackage{amsmath}
\usepackage{mathrsfs}
\usepackage{amsthm}
\usepackage{txfonts}
\usepackage{color}

\usepackage{indentfirst}

\usepackage{anysize}

\usepackage{dsfont}

\usepackage{enumerate}

\allowdisplaybreaks

\newtheorem{theorem}{Theorem}[section]
\newtheorem{lemma}[theorem]{Lemma}
\newtheorem{corollary}[theorem]{Corollary}
\newtheorem{proposition}[theorem]{Proposition}
\theoremstyle{definition}
\newtheorem{remark}[theorem]{Remark}

\newtheorem{definition}[theorem]{Definition}

\numberwithin{equation}{section}

\begin{document}

\title{\Large\bf Equivalent Characterizations and Their Applications
of Solvability of $L^p$ Poisson--Robin(-Regularity)
Problems on Rough Domains
\footnotetext{\hspace{-0.35cm} 2020 {\it Mathematics Subject
Classification}. {Primary 35J25; Secondary 35J15, 35J05, 42B35, 42B25, 35B65.}
\endgraf{\it Key words and phrases}.  $L^p$ Robin problem, Poisson--Robin problem,
Poisson--Robin-regularity problem, $L^p$ Neumann problem, $L^p$ Dirichlet problem, solvability.
\endgraf This project is partially supported by the National Key Research
and Development Program of China (Grant No. 2020YFA0712900), the National Natural
Science Foundation of China (Grant Nos. 12431006 and 12371093),
the Key Project of Gansu Provincial National Science Foundation (Grant No. 23JRRA1022),
Longyuan Young Talents of Gansu Province, and the Fundamental Research Funds
for the Central Universities (Grant No. 2233300008).}}
\author{Xuelian Fu, Dachun Yang\footnote{Corresponding author, E-mail:
\texttt{dcyang@bnu.edu.cn}/{\color{red}\today}/Final version.}\ \ and Sibei Yang}
\date{}
\maketitle

\vspace{-0.8cm}

\begin{center}
\begin{minipage}{13.5cm}\small
{{\bf Abstract.}  Let $n\ge2$, $\Omega\subset\mathbb{R}^n$ be a bounded one-sided chord arc domain,
and $p\in(1,\infty)$. In this article, we study the (weak) $L^p$
Poisson--Robin(-regularity) problem for a uniformly elliptic operator
$L:=-\mathrm{div}(A\nabla\cdot)$ of divergence form on $\Omega$,
which considers weak solutions to the equation
$Lu=h-\mathrm{div}\boldsymbol{F}$ in $\Omega$ with the Robin boundary condition $A\nabla u\cdot\boldsymbol{\nu}+\alpha u=\boldsymbol{F}\cdot\boldsymbol{\nu}$
on the boundary $\partial\Omega$ for functions $h$ and $\boldsymbol{F}$ in some tent spaces.
Precisely, we establish several equivalent characterizations of the
solvability of the (weak) $L^p$ Poisson--Robin(-regularity)
problem and clarify the relationship between the $L^p$ Poisson--Robin(-regularity) problem and the classical $L^p$ Robin problem. Moreover, we also give an extrapolation
property for the solvability of the classical $L^p$ Robin problem.
As applications, we further prove that, for the Laplace operator $-\Delta$ on the bounded Lipschitz
domain $\Omega$, the $L^p$ Poisson--Robin and the $L^q$ Poisson--Robin-regularity problems are respectively solvable for $p\in(2-\varepsilon_1,\infty)$
and $q\in(1,2+\varepsilon_2)$, where $\varepsilon_1\in(0,1]$ and
$\varepsilon_2\in(0,\infty)$ are constants depending only on $n$ and the Lipschitz
constant of $\Omega$ and, moreover, these ranges $(2-\varepsilon_1,\infty)$ of $p$ and $(1,2+\varepsilon_2)$ of $q$ are sharp.
The main results in this article are the analogues of the corresponding
results of both the $L^p$ Poisson--Dirichlet(-regularity) problem, established by
M. Mourgoglou, B. Poggi and X. Tolsa [J. Eur. Math. Soc. 2025], and of the $L^p$ Poisson--Neumann(-regularity) problem,
established by J. Feneuil and L. Li [arXiv: 2406.16735], in the Robin case.}
\end{minipage}
\end{center}

\vspace{0.1cm}

\section{Introduction and main results\label{s1}}

Let $n\ge2$, $\Omega\subset\mathbb{R}^n$ be a bounded one-sided chord arc domain,
and $p\in(1,\infty)$. In this article, motivated by the recent work \cite{mpt22,fl24} on the solvability of
the $L^p$ Poisson--Dirichlet(-regularity) problem and the $L^p$ Poisson--Neumann(-regularity) problem for a uniformly elliptic operator
$L:=-\mathrm{div}(A\nabla\cdot)$ of divergence form on $\Omega$,
we study the solvability of the $L^p$ Poisson--Robin(-regularity) problem
for $L$ on $\Omega$ and clarify the relationship between the $L^p$ Poisson--Robin(-regularity) problem and the classical $L^p$ Robin problem.
As applications, we further obtain, for the special case of $L:=-\Delta$ being the Laplace operator on the bounded Lipschitz domain $\Omega$,
the $L^p$ Poisson--Robin and the $L^q$ Poisson--Robin-regularity problems are respectively solvable for any $p\in(2-\varepsilon_1,\infty)$
and $q\in(1,2+\varepsilon_2)$, where $\varepsilon_1\in(0,1]$ and $\varepsilon_2\in(0,\infty)$ are constants depending only on $n$
and the Lipschitz constant of $\Omega$ and, moreover, these ranges
$(2-\varepsilon_1,\infty)$ of $p$ and $(1,2+\varepsilon_2)$ of $q$ are sharp.

The $L^p$ solvability of homogeneous boundary value problems
for second-order linear elliptic partial
differential equations with non-smooth coefficients on
rough domains developed rapidly in the last few decades.
Indeed, since the initial work of Dahlberg \cite{d77} on
the relationship between harmonic measures and surface measures on
the boundary of Lipschitz domains in $\mathbb{R}^n$, there emerged a lot
of significant works on various boundary value problems for
second-order elliptic operators $L:=-\mathrm{div\,}(A\nabla\cdot)$
with $L^p$ boundary data (see, for example, \cite{ahmmt20,dk87,jk81a,k94,kp93,mpt22}),
where most of the work is related to the Dirichlet problem. Precisely, there exist
various characterizations for the solvability of
the $L^p$ Dirichlet problem (see, for example, \cite{ahmmt20,fjr78,hl18,hlm19,jk81a,jk82,kkpt00,v84}), and simultaneously there exist a large amount of literatures
regarding the solvability of the $L^p$ Dirichlet problem for elliptic operators
of divergence form with non-smooth coefficients on rough domains (see, for example, \cite{aaahk11,chmt20,dj90,dp19,dpp07,fp22,hkmp15a,hlmp22,kp01}).
In particular, the necessary and sufficient geometric conditions
on domains for guaranteeing the solvability of the $L^p$ Dirichlet problem
is one of the core issues (see, for example, \cite{ahmmt20,gmt18,hmm16,hmmtz21,mt22}).

Compared with the Dirichlet case, the Neumann problem with $L^p$
boundary data remains less understood (see, for example, \cite{fl24,hs25,k94,mt24}).
To deal with the solvability of the $L^p$ Neumann problem, one typically needs to impose more restrictive assumptions
on both the operator and the domains (see, for example, \cite{dhp23,dpr17,jk81b,kp93,kp96,kr09,ks08,mt24}).
It is remarkable to point out that it is a long standing open problem to solve the $L^p$ Neumann problem in chord arc domains or more general domains
for the Laplace operator or more general elliptic operators (see, for example,
\cite{k94,kp93,kp96,mt24}).

Moreover, as a natural extension of the Dirichlet and the Neumann boundary value problems, the Robin type boundary value problem is ubiquitous
in physics and biology (see, for example, \cite{ddemm24,gfs06,w84}). The solvability of the $L^p$ Robin problem for the Laplace operator on
Lipschitz domains was obtained by Lanzani and Shen \cite{ls05} (see also \cite{yyy18}).
Meanwhile, the global Sobolev regularity for the Robin type boundary
value problem on Lipschitz domains or more general rough domains was studied
by Dong and Li \cite{dl21} and the second and the third authors of this article and
Yuan \cite{yyy21} (see also \cite{yyz24}).

Let $p\in(1,\infty)$ and $p'$ be the H\"older conjugate of $p$, that is, $\frac1p+\frac{1}{p'}=1$. Compared with the $L^p$ solvability of homogeneous
boundary value problems, the literatures on the $L^p$ solvability of corresponding inhomogeneous problems are far more scant. Recently,
in order to seek an effective approach to solve the regularity problem for second-order elliptic operators with non-smooth
coefficients on rough domains, Mourgoglou et al. \cite{mpt22} seminally studied the $L^p$
solvability of the following inhomogeneous Dirichlet problem on the one-sided chord arc domain $\Omega$:
\begin{equation}\label{e1.1}
\begin{cases}
-\mathrm{div\,}(A\nabla u)=h-\mathrm{div\,}\boldsymbol{F}\ \ & \text{in}\ \ \Omega,\\
 u=f \ \ & \text{on}\ \ \partial\Omega,\\
 \mathcal{N}(u)\in L^p(\partial\Omega),
\end{cases}
\end{equation}
which is called the \emph{$L^p$ Poisson--Dirichlet problem},
where $\mathcal{N}$ denotes the non-tangential maximal function (see Definition \ref{NAC}). Precisely, Mourgoglou et al. \cite{mpt22}
proved that the solvability of the $L^p$ Poisson--Dirichlet problem
is equivalent to the solvability of the homogeneous $L^{p'}$ Dirichlet problem, which is defined by setting $h\equiv0$ and
$\boldsymbol{F}=\boldsymbol{0}$ in \eqref{e1.1}. Moreover, inspired by the work of Mourgoglou et al. \cite{mpt22},
Feneuil and Li \cite{fl24} studied the $L^p$ solvability of the inhomogeneous Neumann problem, which is called the
\emph{$L^p$ Poisson--Neumann problem}, defined via letting $h\equiv0$ and replacing the Dirichlet boundary condition $u=f$
by the Neumann boundary condition $\frac{\partial u}{\partial\boldsymbol{\nu}}=\boldsymbol{F}\cdot\boldsymbol{\nu}$ in \eqref{e1.1}.
To be more precise, Feneuil and Li \cite{fl24} proved that the solvability of the $L^{p}$ Poisson--Neumann problem
implies the solvability of the homogeneous $L^{p'}$ Neumann problem, and the solvability of the homogeneous $L^{p'}$ Neumann problem implies
the solvability of the $L^{p}$ Poisson--Neumann problem under the additional assumption that the homogeneous $L^{p}$ Dirichlet problem is solvable.
Meanwhile, by obtaining an extrapolation theorem for the solvability of the $L^p$ Poisson--Neumann problem, Feneuil and Li \cite{fl24}
obtained a new extrapolation theorem for the solvability of the homogeneous $L^p$ Neumann problem, which improves the extrapolation theorem
for the homogeneous $L^p$ Neumann problem established by Kenig and Pipher \cite{kp93}.

To state the main results of this article, we start with presenting
some necessary symbols.

\begin{definition}
Let $n\ge2$. A Borel measure $\sigma$ on $\mathbb{R}^n$ is said to be $(n-1)$-\emph{Ahlfors regular} if there
exists a positive constant $C$ such that, for any $r\in(0,\mathrm{diam\,}(\mathrm{\,supp}\sigma))$ and $x\in\mathrm{supp\,}\sigma$,
$$
C^{-1}r^{n-1}\le C\sigma(B(x,r))\le Cr^{n-1}.
$$
Here, and thereafter, $B(x,r):=\{y\in\mathbb{R}^n:|x-y|<r\}$, $$\mathrm{supp\,}\sigma:=\left\{y\in\mathbb{R}^n:\sigma(B(y,s))>0\ \text{for any}\  s\in(0,\infty)\right\},$$
and, for any measurable set $E\subset\mathbb{R}^n$, $\mathrm{diam\,}(E):=\sup\{|x-y|: x,y\in E\}$.

Moreover, a closed set $E\subset\mathbb{R}^n$ is said to be \emph{$(n-1)$-Ahlfors regular} if $\mathcal{H}^{n-1}|_E$
is $(n-1)$-Ahlfors regular, where $\mathcal{H}^{n-1}$ denotes the $(n-1)$-dimensional Hausdorff measure.
\end{definition}

\begin{definition}
Let $n\ge2$ and $\Omega\subset\mathbb{R}^n$ be a domain.
Then $\Omega$ is called a \emph{one-sided non-tangentially accessible} (for short, NTA) \emph{domain} if it satisfies
the corkscrew and the Harnack chain conditions (see Definition \ref{CSandHC}).

Furthermore, $\Omega$ is called a \emph{one-sided chord arc domain}
(for short, CAD) if $\Omega$ is a one-sided NTA domain, $\partial\Omega$ is $(n-1)$-Ahlfors regular, and $\Omega$ is bounded
whenever $\partial\Omega$ is bounded.
Meanwhile, denote $\mathcal{H}^{n-1}|_{\partial\Omega}$ by $\sigma$ and call it the \emph{surface measure} on $\partial\Omega$.
\end{definition}

\begin{definition}\label{NAC}
Let $n\ge2$ and $\Omega\subset\mathbb{R}^n$ be a one-sided CAD. Let $a\in(1,\infty)$
and $c\in(0,1)$. Denote by $L^2_{\mathrm{loc}}(\Omega)$ the \emph{set} of all measurable functions $f$ on $\Omega$ satisfying
$\int_{K}|f(x)|^2\,dx<\infty$ for any compact set $K\subset\Omega$.
\begin{enumerate}[(a)]
\item Let $u$ be a measurable function on $\Omega$.
The \emph{non-tangential maximal function} $\mathcal{N}^a(u)$ of $u$ is defined by setting, for any $x\in\partial\Omega$,
$$
\mathcal{N}^a(u)(x):=\sup_{y\in\gamma_{a}(x)}|u(y)|,
$$
where the \emph{cone} $\gamma_{a}(x)$ with vertex $x$ and aperture $a$
is defined by setting
$$
\gamma_{a}(x):=\left\{z\in\Omega:|z-x|<a\delta(z)\right\}.
$$
Here, and thereafter, for any $z\in\Omega$, $\delta(z):=\text{dist\,}(z,\partial\Omega):=\inf\{|z-y|: y\in\partial\Omega\}$.
\item Let $u\in L_{\mathrm{loc}}^2(\Omega)$. The \emph{modified non-tangential
maximal function} $\widetilde{\mathcal{N}}^{a,c}(u)$ of $u$ is defined by
setting, for any $x\in\partial\Omega$,
$$
\widetilde{\mathcal{N}}^{a,c}(u)(x):=\sup_{y\in\gamma_{a}(x)}
\left[\fint_{B(y,c\delta(y))}|u(z)|^2\,dz\right]^{\frac{1}{2}}.
$$
Here, and thereafter, for any locally integrable function $f$ on $\mathbb{R}^n$,
$$\fint_{B}f(z)\,dz:
=\frac{1}{|B|}\int_{B}f(z)\,dz.$$

\item Let $F\in L_{\mathrm{loc}}^2(\Omega)$. The \emph{Lusin area function}
$\widetilde{\mathcal{A}}_{1}^{a,c}(F)$ of $F$ is defined by setting, for any $x\in\partial\Omega$,
$$
\widetilde{\mathcal{A}}_{1}^{a,c}(F)(x):=\int_{\gamma_{a}(x)}
\left[\fint_{B(y,c\delta(y))}|F(z)|^2\,dz\right]^{\frac{1}{2}}\,\frac{dy}{[\delta(y)]^{n}}.
$$

\item Let $F\in L_{\mathrm{loc}}^2(\Omega)$. The \emph{averaged Carleson functional}
$\widetilde{\mathcal{C}}_{1}^{c}(F)$ of $F$ is defined by setting, for any $x\in\partial\Omega$,
$$
\widetilde{\mathcal{C}}_{1}^{c}(F)(x):=
\sup_{r\in(0,\mathrm{dist\,}(\partial\Omega))}
\frac{1}{\sigma(B(x,r))}\int_{B(x,r)\cap\Omega}\left[\fint_{B(y,c\delta(y))}
|F(z)|^2\,dz\right]^{\frac{1}{2}}\frac{dy}{\delta(y)}.
$$
\end{enumerate}
In particular, let
$$
\gamma(x):=\gamma_2(x), \ \ \widetilde{\mathcal{N}}:=\widetilde{\mathcal{N}}^{2,1/4}, \ \ \widetilde{\mathcal{A}}_1:
=\widetilde{\mathcal{A}}_1^{2,1/4}, \ \text{and} \ \ \widetilde{\mathcal{C}}_1:=\widetilde{\mathcal{C}}_1^{1/4}.
$$
\end{definition}
It is worth pointing out that the $L^p$ norms of the modified non-tangential maximal function, the Lusin area function, and the averaged
Carleson functional are respectively equivalent under the change of the aperture $a$ or the radius $c$ (see Lemma \ref{TentEqui}).
Moreover, when we do not need to specify neither $\alpha $ nor $c$, we omit the corresponding superscript.

\begin{definition}
Let $\Omega\subset\mathbb{R}^n$ be a domain. For any given $x\in\Omega$, let $A(x):=\{a_{ij}(x)\}_{i,j=1}^n$ denote an $n\times n$
matrix with real-valued, bounded, and measurable entries.
Then $A$ is said to satisfy the \emph{uniform ellipticity condition} if there exists a positive constant
$\mu_0\in(0,1]$ such that, for any $x\in \Omega$ and $\xi:=(\xi_1,\cdots,\xi_n)\in\mathbb{R}^n$,
\begin{equation}\label{e1.2}
\mu_0 |\xi|^2\le\sum_{i,j=1}^na_{ij}(x)\xi_i\xi_j\le\mu_0^{-1} |\xi|^2.
\end{equation}
Meanwhile, the operator $L:=-\mathrm{div}(A\nabla\cdot)$ is said to be \emph{uniformly elliptic} if the matrix $A$ satisfies \eqref{e1.2}.
Moreover, denote by $L^*:=-\mathrm{div}(A^T\nabla\cdot)$ the \emph{adjoint operator}
of $L$, where $A^T$ is the \emph{transpose} of $A$.
\end{definition}

Let $p\in[1,\infty)$. Denote by $W^{1,p}(\Omega)$ the \emph{Sobolev space} on $\Omega$ equipped with the \emph{norm}
$$\|f\|_{W^{1,p}(\Omega)}:=\|f\|_{L^p(\Omega)}+\|\,|\nabla f|\,\|_{L^p(\Omega)},$$
where $\nabla f$ denotes the \emph{distributional gradient} of $f$.
Furthermore, $W^{1,p}_{0}(\Omega)$ is defined to be the \emph{closure} of $C^{\infty}_{\mathrm{c}} (\Omega)$ in $W^{1,p}(\Omega)$, where
$C^{\infty}_{\mathrm{c}}(\Omega)$ denotes the set of all \emph{infinitely differentiable functions on $\Omega$ with compact support contained in $\Omega$}.

Now, we give the definition of weak solutions to the Robin problem.

\begin{definition}
Let $n\ge2$ and $\Omega\subset\mathbb{R}^n$ be a bounded one-sided CAD. Assume that $\alpha$
is a measurable function on $\partial\Omega$ satisfying
\begin{equation}\label{e1.3}
0\le\alpha\in L^{n-1+\varepsilon_0}(\partial\Omega,\sigma)\ \ \text{and}\ \ \alpha\ge c_0\ \ \text{on}\ \ E_0\subset\partial\Omega,
\end{equation}
where $\varepsilon_0,c_0\in(0,\infty)$ are given constants and the measurable set $E_0$ satisfies $\sigma(E)>0$.

Let $L:=-\mathrm{div}(A\nabla\cdot)$ be a uniformly elliptic operator, $h\in L^2(\Omega)$, $\boldsymbol{F}\in L^2(\Omega;\mathbb{R}^n)$,
and $f\in L^2(\partial\Omega,\sigma)$.
Then $u\in W^{1,2}(\Omega)$ is called a \emph{weak solution} to the \emph{Robin problem}
\begin{equation}\label{e1.4}
\begin{cases}
Lu=h-\mathrm{div\,}\boldsymbol{F}\ \ & \text{in}\ \ \Omega,\\
\displaystyle\frac{\partial u}{\partial\boldsymbol{\nu}}+\alpha  u=f+\boldsymbol{F}\cdot\boldsymbol{\nu} \ \ & \text{on}\ \ \partial\Omega,
\end{cases}
\end{equation}
where $\boldsymbol{\nu}:=(\nu_1,\ldots,\nu_n)$ denotes the \emph{outward unit normal}
to $\partial\Omega$ and $\frac{\partial u}{\partial\boldsymbol{\nu}}:=(A\nabla u)\cdot\boldsymbol{\nu}$ denotes the \emph{conormal derivative} of $u$ on $\partial\Omega$,
if, for any $\phi\in C_{\rm c}^{\infty}(\mathbb{R}^n)$ (the set of all infinitely differentiable
functions on $\mathbb{R}^n$ with compact support),
\begin{equation}\label{e1.5}
\int_{\Omega}A\nabla u\cdot\nabla\phi\,dx+\int_{\partial\Omega}\alpha  u\phi\,d\sigma
=\int_{\partial\Omega}f\phi\,d\sigma+\int_{\Omega}(h\phi+\boldsymbol{F}\cdot\nabla\phi)\,dx.
\end{equation}
\end{definition}

\begin{remark}\label{r1.1}
We point out that, using the Lax--Milgram theorem,
one can find that there exists a unique weak solution
$u\in W^{1,2}(\Omega)$ to the Robin problem \eqref{e1.4} (see \cite{ddemm24,wy25} for the details). Moreover, when $f\equiv0$
in \eqref{e1.4}, we further have
$$
\|u\|_{W^{1,2}(\Omega)}\le C\left[\|h\|_{L^2(\Omega)}+\|\boldsymbol{F}\|_{L^2(\Omega;\mathbb{R}^n)}\right],
$$
where $C$ is a positive constant independent of $u$, $h$, and $\boldsymbol{F}$.
\end{remark}

Next, following the definition of the solvability of the $L^p$ Neumann problem \cite{kp93,kp96,fl24}, we give the definition of
the solvability of the $L^p$ Robin problem.

\begin{definition}\label{defR}
Let $n\ge2$, $\Omega\subset\mathbb{R}^n$ be a bounded one-sided CAD, $p\in(1,\infty)$,
$L:=-\mathrm{div}(A\nabla\cdot)$ be a uniformly elliptic operator,
and $\alpha$ be the same as in \eqref{e1.3}. The $L^p$ Robin problem
$(\mathrm{R}_p)_{L}$ is said to be \emph{solvable} if there exists
a positive constant $C$ such that, for any $f\in C_{\rm c}(\partial\Omega)$
(the set of all continuous
functions on $\partial\Omega$ with compact support), the weak solution $u$ to the Robin problem
\begin{equation}\label{e1.6}
\begin{cases}
Lu=0 \ \ & \text{in}\ \ \Omega,\\
\displaystyle\frac{\partial u}{\partial\boldsymbol{\nu}}+\alpha  u=f \ \ & \text{on}\ \ \partial\Omega
\end{cases}
\end{equation}
satisfies the estimate
$$
\left\|\widetilde{\mathcal{N}}(\nabla u)\right\|_{L^{p}(\partial \Omega,\sigma)}\le C \|f\|_{L^{p}(\partial \Omega,\sigma)}.
$$

Moreover, the $L^p$ Robin problem $(\widetilde{\mathrm{R}}_{p})_{L}$
is said to be \emph{solvable} if there exists a positive
constant $C$ such that, for any $f\in C_{\rm c}(\partial\Omega)$,
the weak solution $u$ to \eqref{e1.6} satisfies
$$
\left\|\widetilde{\mathcal{N}}(\nabla u)\right\|_{L^{p}(\partial \Omega,\sigma)}+\left\|\widetilde{\mathcal{N}}(u)\right\|_{L^{p}(\partial \Omega,\sigma)}
\le C \|f\|_{L^{p}(\partial \Omega,\sigma)}.
$$
\end{definition}

\begin{remark}
We point out that the solvability of both the Robin problems $(\mathrm{R}_p)_{L}$ and $(\widetilde{\mathrm{R}}_{p})_{L}$ is equivalent
(see Proposition \ref{RobinEquiv}).
\end{remark}

Denote by $L_{\rm c}^{\infty}(\Omega)$ the set of all bounded measurable
functions on $\Omega$ with compact support.
Following \cite{mpt22,fl24}, we introduce the $L^p$ (weak)
Poisson--Robin problem and the $L^p$ (weak) Poisson--Robin-regularity problem.

\begin{definition}\label{defPR}
Let $n\ge2$, $\Omega\subset\mathbb{R}^n$ be a bounded one-sided {\rm CAD}, $p\in(1,\infty)$,
$L:=-\mathrm{div}(A\nabla\cdot)$ be a uniformly elliptic operator,
and $\alpha$ be the same as in \eqref{e1.3}. The $L^p$ Poisson--Robin
problem $(\mathrm{PR}_{p})_{L}$ is said to be \emph{solvable}
if there exists a positive constant $C$ such that, for any $\boldsymbol{F}\in L_{\rm c}^{\infty}(\Omega;\mathbb{R}^n)$,
the weak solution $u$ to the Robin problem
\begin{equation}\label{e1.7}
\begin{cases}
Lu=-\mathrm{div\,}\boldsymbol{F} \ \ & \text{in}\ \ \Omega,\\
\displaystyle\frac{\partial u}{\partial\boldsymbol{\nu}}+\alpha  u=\boldsymbol{F}\cdot\boldsymbol{\nu} \ \ & \text{on}\ \ \partial\Omega
\end{cases}
\end{equation}
satisfies
$$
\left\|\widetilde{\mathcal{N}}(u)\right\|_{L^{p}(\partial \Omega,\sigma)} \leq C
\left\|\widetilde{\mathcal C}_1(\delta\boldsymbol{F})\right\|_{L^{p}(\partial \Omega,\sigma)},
$$
where, for any $x\in\Omega$, $\delta(x):=\mathrm{dist\,}(x,\partial\Omega)$.

Moreover, the weak $L^p$ Poisson--Robin problem
$(\mathrm{wPR}_{p})_{L}$ is said to be \emph{solvable}
if there exists a positive constant $C$ such that, for any $\boldsymbol{F}\in L_{\rm c}^{\infty}(\Omega;\mathbb{R}^n)$,
the weak solution $u$ to the Robin problem \eqref{e1.7} satisfies
$$
\left\|\widetilde{\mathcal{N}}(\delta\nabla u)\right\|_{L^{p}(\partial \Omega,\sigma)} \le C
\left\|\widetilde{\mathcal C}_1(\delta\boldsymbol{F})\right\|_{L^{p}(\partial \Omega,\sigma)}.
$$
\end{definition}

The next definition presents the Poisson--Robin-regularity problem, which is the dual formulation of the Poisson--Robin problem
in some sense.

\begin{definition}\label{defPRR}
Let $n\ge2$, $\Omega\subset\mathbb{R}^n$ be a bounded one-sided {\rm CAD}, $p\in(1,\infty)$,
$L:=-\mathrm{div}(A\nabla\cdot)$ be a uniformly elliptic operator,
and $\alpha$ be the same as in \eqref{e1.3}. The $L^p$ Poisson--Robin-regularity problem $(\mathrm{PRR}_{p})_{L}$ is said to be \emph{solvable}
if there exists a positive constant $C$ such that, for any $h\in L_{\rm c}^{\infty}(\Omega)$ and
$\boldsymbol{F}\in L_{\rm c}^{\infty}(\Omega;\mathbb{R}^n)$, the weak solution $u$ to the Robin problem
\begin{equation}\label{e1.8}
\begin{cases}
Lu=h-\mathrm{div\,}\boldsymbol{F} \ \ & \text{in}\ \ \Omega,\\
\displaystyle\frac{\partial u}{\partial\boldsymbol{\nu}}+\alpha  u=\boldsymbol{F}
\cdot\boldsymbol{\nu} \ \ & \text{on}\ \ \partial\Omega
\end{cases}
\end{equation}
satisfies
$$
\left\|\widetilde{\mathcal{N}}(\nabla u)\right\|_{L^{p}(\partial \Omega,\sigma)} \le C \left\|
\widetilde{\mathcal C}_1(\delta|h|+|\boldsymbol{F}|)\right\|_{L^{p}(\partial \Omega,\sigma)}.
$$
Moreover, the weak $L^p$ Poisson--Robin-regularity
problem $(\mathrm{wPRR}_{p})_{L}$ is said to be \emph{solvable}
if there exists a positive constant $C$ such that, for any $\boldsymbol{F}\in L_{\rm c}^{\infty}(\Omega;\mathbb{R}^n)$, the weak solution $u$
to the Robin problem \eqref{e1.8} with $h\equiv0$ satisfies
$$
\left\|\widetilde{\mathcal{N}}(\nabla u)\right\|_{L^{p}(\partial \Omega,\sigma)} \le C
\left\|\widetilde{\mathcal C}_1(\boldsymbol{F})\right\|_{L^{p}(\partial \Omega,\sigma)}.
$$
\end{definition}

\begin{remark}\label{r1.2}
Assume that $\alpha\equiv0$ and $\int_{\partial\Omega} f\,d\sigma=0$ in \eqref{e1.6}. Then the $L^p$ Robin problem $({\rm R}_p)$
is precisely the $L^p$ Neumann problem, denoted by $({\rm N}_p)$
(see, for example, \cite{fl24,kp93}).
Similarly, letting $\alpha\equiv0$ in \eqref{e1.7} and \eqref{e1.8}, then the $L^p$ Poisson--Robin problem $(\mathrm{PR}_{p})_{L}$
and the $L^p$ Poisson--Robin-regularity problem $(\mathrm{PRR}_{p})_{L}$ are exactly respectively the $L^p$ Poisson--Neumann problem $(\mathrm{PN}_{p})_{L}$
and the $L^p$ Poisson--Neumann-regularity problem $(\mathrm{PNR}_{p})_{L}$ introduced in \cite{fl24}.
\end{remark}

The first result of this article is the following equivalence between the (weak) Poisson--Robin problem and the (weak) Poisson--Robin-regularity problem,
which can be regarded as a natural analogue
of \cite[Theorem 1.22]{mpt22} and \cite[Propostion 1.17]{fl24}
in the Robin case.

\begin{theorem}\label{PRandPRR}
Let $n\ge2$, $\Omega\subset\mathbb{R}^n$ be a bounded one-sided {\rm CAD}, $p\in(1,\infty)$, $p'$ be the H\"older conjugate of $p$,
$L:=-\mathrm{div}(A\nabla\cdot)$ be a uniformly elliptic operator, and $\alpha$ be the same as in \eqref{e1.3}.
\begin{enumerate}[{\rm(i)}]
\item The following assertions are mutually equivalent.
\begin{enumerate}[{\rm(a)}]
\item The Poisson--Robin problem $({\rm PR}_{p'})_{L^*}$ is solvable;
\item The Poisson--Robin-regularity problem $({\rm PRR}_{p})_L$ is solvable;
\item The Poisson--Robin-regularity problem $({\rm PRR}_{p})_L$ is solvable for $\boldsymbol{F}=\mathbf{0}$.
\end{enumerate}

\item The following assertions are equivalent.
\begin{enumerate}
\item[{\rm(d)}] The weak Poisson--Robin problem $({\rm wPR}_{p'})_{L^*}$ is solvable;
\item[{\rm(e)}] The weak Poisson--Robin-regularity problem $({\rm wPRR}_{p})_L$ is solvable.
 \end{enumerate}
\end{enumerate}
\end{theorem}

We prove Theorem \ref{PRandPRR} by applying the duality between the modified non-tangential maximal function $\widetilde{\mathcal{N}}$ and the averaged
Carleson functional $\widetilde{\mathcal{C}}_1$ (see Lemma \ref{NdualC}).

Furthermore, one can find that the weak Poisson--Robin problem is
indeed weaker than the Poisson--Robin problem by definition.
In addition, we give a localization property, which can serve
as a bridge connecting the solvability of
the Poisson--Robin problem and the weak Poisson--Robin problem.
Similar local properties were given in \cite{mpt22} and \cite{fl24}, respectively,
for the Dirichlet and the Neumann problems.

\begin{definition}\label{defLoc}
Let $n\ge2$, $\Omega\subset\mathbb{R}^n$ be a bounded one-sided {\rm CAD}, $p\in(1,\infty)$, $L:=-\mathrm{div}(A\nabla\cdot)$ be a uniformly elliptic operator,
and $\alpha$ be the same as in \eqref{e1.3}. The \emph{local property}
$({\rm LocR}_p)_L$ is said to hold if there exists a positive constant
$C$ such that, for any ball $B:=B(x,r)$ with $x\in\partial\Omega$ and $r\in(0,\mathrm{diam\,}(\Omega))$ and for any local solution
$u\in W^{1,2}(2B\cap\Omega)$ to $Lu=0$ in $2B\cap\Omega$ with the Robin boundary condition $\frac{\partial u}{\partial\boldsymbol{\nu}}+\alpha  u=0$
on $2B\cap\partial\Omega$, it holds that
\begin{equation}\label{e1.9}
\left\|\widetilde{\mathcal{N}}(|\nabla u|\mathbf{1}_B)\right\|_{L^{p}(\partial \Omega,\sigma)}
\le C r^{\frac{n-1}{p}}\left[\fint_{2B\cap\Omega}|\nabla u|\,dx+\fint_{2B\cap\Omega}|u|\,dx\right],
\end{equation}
where $\mathbf{1}_{B}$ denotes the \emph{characteristic function} of $B$.
\end{definition}

\begin{remark}
Let $\Omega$ and $L:=-\mathrm{div}(A\nabla\cdot)$ be the same
as in Definition \ref{defLoc} and $p\in(1,\infty)$.
Feneuil and Li \cite[Definition 1.15]{fl24} introduced the local property
$({\rm LocN}_p)_{L}$ for the Neumann problem.
Precisely, the \emph{local property} $({\rm LocN}_p)_{L}$ is said to hold if,
for any ball $B:=B(x,r)$ with $x\in\partial\Omega$ and
$r\in(0,\mathrm{diam\,}(\Omega))$ and for any local solution
$u\in W^{1,2}(2B\cap\Omega)$ to $Lu=0$ in $2B\cap\Omega$ with the condition $\frac{\partial u}{\partial\boldsymbol{\nu}}=0$
on $2B\cap\partial\Omega$, it holds that
\begin{equation}\label{e1.10}
\left\|\widetilde{\mathcal{N}}(|\nabla u|\mathbf{1}_B)\right\|_{L^{p}(\partial \Omega,\sigma)}
\le C r^{\frac{n-1}{p}}\fint_{2B\cap\Omega}|\nabla u|\,dx.
\end{equation}

Indeed, in the Neumann case, using Poincar\'e's inequality,
one can easily prove that the local properties \eqref{e1.9}
and \eqref{e1.10} are equivalent. However, in the Robin case
\eqref{e1.9} and \eqref{e1.10} may not be equivalent,
and the appropriate version of the local property for the Robin problem
is \eqref{e1.9} (see, for example, \cite{dl21,yyy21}).
\end{remark}

Then we have the following conclusions.

\begin{theorem}\label{StrongLocWeak}
Let $n\ge2$, $\Omega\subset\mathbb{R}^n$ be a bounded one-sided {\rm CAD}, $p\in(1,\infty)$, $p'$ be the H\"older conjugate of $p$,
$L:=-\mathrm{div}(A\nabla\cdot)$ be a uniformly elliptic operator, and $\alpha$ be the same as in \eqref{e1.3}.
Then the following statements hold.
\begin{enumerate}[{\rm(i)}]
\item  Assume further that $\alpha\in L^p(\partial\Omega,\sigma)$ and $\|\alpha \|_{L^p(B(x,r)\cap\partial\Omega,\sigma)}
\le Cr^{(n-1)/p}$ for any $x\in\partial\Omega$ and $r\in(0,\mathrm{diam\,}(\Omega))$, where $C$ is a positive constant independent of $x$ and $r$.
If $({\rm PR}_{p'})_{L^*}$ is solvable, then $({\rm LocR}_p)_L$ holds.

\item  If $({\rm LocR}_p)_L$ holds, then there exists a positive constant $\varepsilon$, depending on $p$, $L$, and
$\Omega$, such that $({\rm LocR}_{q})_L$ holds for any $q\in[1,p+\varepsilon)$.

\item  If $n\ge3$ and $({\rm LocR}_p)_L$ holds, then $({\rm wPR}_{p'})_{L^*}$ is solvable.
\end{enumerate}
\end{theorem}

We prove Theorem \ref{StrongLocWeak} by applying the duality between the modified non-tangential maximal function $\widetilde{\mathcal{N}}$ and the averaged
Carleson functional $\widetilde{\mathcal{C}}_1$, Caccioppoli's inequality, the Moser type estimate, the boundary H\"older regularity estimate
for weak solutions to the Robin problem \eqref{e1.4}, and the properties of Green's function associated with the Robin problem.

Next, we present the relationship among the solvability of the
problems $({\rm R}_p)_L$, $({\rm PR}_{p'})_{L^*}$,
and $({\rm PRR}_p)_L$. To do this, we first recall the definition of the $L^p$ Dirichlet problem (see, for example, \cite{mpt22}).

\begin{definition}\label{defD}
Let $n\ge2$, $\Omega\subset\mathbb{R}^n$ be a one-sided {\rm CAD}, $p\in(1,\infty)$, and $L:=-\mathrm{div}(A\nabla\cdot)$
be a uniformly elliptic operator.
The \emph{$L^p$ Dirichlet problem $({\rm D}_{p})_L$ is said to be solvable}
if there exists a positive constant $C$ such that, for any $f\in C_{\rm c}(\partial\Omega)$, the solution $u$ to the Dirichlet problem
$$
\begin{cases}
Lu=0 \ \ & \text{in}\ \ \Omega,\\
u=f \ \ & \text{on}\ \ \partial\Omega
\end{cases}
$$
satisfies
$$
\left\|\widetilde{\mathcal{N}}(u)\right\|_{L^{p}(\partial \Omega,\sigma)} \le C \|f\|_{L^{p}(\partial \Omega,\sigma)}.
$$
Moreover, by replacing the boundary condition $\frac{\partial u}{\partial\boldsymbol{\nu}}+\alpha  u=\boldsymbol{F}\cdot\boldsymbol{\nu}$
in \eqref{e1.7} and \eqref{e1.8} with the Dirichlet boundary condition $u=0$
on $\partial\Omega$, and following the definitions used for the solvability
of the $L^p$ Poisson--Robin problem $({\rm PR}_{p})_L$ and the $L^p$ Poisson--Robin-regularity problem $({\rm PRR}_{p})_L$,
we can define the solvability of the corresponding \emph{$L^p$ Poisson--Dirichlet problem} $({\rm PD}_{p})_L$ and
\emph{$L^p$ Poisson--Dirichlet-regularity problem} $({\rm PDR}_{p})_L$ (see \cite[Definition 1.17]{mpt22} for the details).
\end{definition}

\begin{theorem}\label{RandPR}
Let $n\ge2$, $\Omega\subset\mathbb{R}^n$ be a bounded one-sided {\rm CAD}, $p\in(1,\infty)$, $p'$ be the H\"older conjugate of $p$,
$L:=-\mathrm{div}(A\nabla\cdot)$ be a uniformly elliptic operator, and $\alpha$ be the same as in \eqref{e1.3}.
Then the following assertions hold.
\begin{enumerate}[{\rm (i)}]
\item  If $({\rm PRR}_{p})_{L}$ is solvable, then $({\rm R}_{p})_{L}$ is solvable.

\item  If both $({\rm R}_{p})_{L}$ and $({\rm D}_{p'})_{L^*}$ are solvable, then $({\rm PRR}_{p})_{L}$ is solvable.

\item  If $({\rm R}_p)_L$ is solvable and $({\rm LocR}_p)_L$ holds, then there exists a positive constant $\varepsilon$, depending on $p$,
$L$, and $\Omega$, such that $({\rm R}_{q})_L$ is solvable for any $q\in[p,p+\varepsilon)$.

\item  If $({\rm PR}_p)_L$ is solvable, then, for any $q\in[p,\infty)$, $({\rm PR}_{q})_{L}$ is solvable.
\end{enumerate}
\end{theorem}

We prove Theorem \ref{RandPR} by first establishing the equivalence of
the solvability of $({\rm R}_p)_L$ and $({\rm \widetilde{R}}_p)_L$
and using the geometric properties of CADs and a good-$\lambda$ argument.

\begin{remark}
Let $\Omega$, $L:=-\mathrm{div}(A\nabla\cdot)$, and the function $\alpha$ be the same as in Theorem \ref{RandPR} and $p\in(1,\infty)$.
\begin{itemize}
\item[{\rm(i)}] From Theorems \ref{PRandPRR} and \ref{RandPR},
we deduce the following relationship
among the problems $({\rm PRR}_{p})_{L}$, $({\rm PR}_{p'})_{L^*}$, and $({\rm R}_{p})_{L}$:
\begin{equation}\label{e1.11}
({\rm R}_{p})_{L}+({\rm D}_{p'})_{L^*}\Rightarrow({\rm PRR}_{p})_{L}\Leftrightarrow
({\rm PR}_{p'})_{L^*}\Rightarrow
({\rm R}_{p})_{L}.
\end{equation}
Recall that, in the Dirichlet case, Mourgoglou et al. \cite[Theorem 1.22]{mpt22} established the following equivalent relationship
among the problems $({\rm D}_{p})_{L}$, $({\rm PD}_{p})_{L}$, and $({\rm PDR}_{p'})_{L^\ast}$:
\begin{equation}\label{e1.12}
({\rm D}_{p})_{L}\Leftrightarrow({\rm PD}_{p})_{L}\Leftrightarrow({\rm PDR}_{p'})_{L^\ast};
\end{equation}
and, in the Neumann case, Feneuil and Li \cite[Propositions 1.17 and 1.18]{fl24} obtained the following relationship
among the problems $({\rm N}_{p})_{L}$, $({\rm PN}_{p'})_{L^*}$, and $({\rm PNR}_{p})_{L}$:
\begin{equation}\label{e1.13}
({\rm N}_{p})_{L}+({\rm D}_{p'})_{L^*}\Rightarrow({\rm PNR}_{p})_{L}\Leftrightarrow
({\rm PN}_{p'})_{L^*}\Rightarrow({\rm N}_{p})_{L}.
\end{equation}
Thus, the conclusion \eqref{e1.11} is a natural analogue of \eqref{e1.12}
and \eqref{e1.13} in the Robin case
and, moreover, \eqref{e1.11} covers \eqref{e1.13} when the function $\alpha\equiv0$.

\item[{\rm(ii)}] Feneuil and Li \cite[Theorem 1.19]{fl24}
proved the following relationship
among the problems $({\rm PN}_{p'})_{L^*}$, $({\rm PNR}_{p})_{L}$, $({\rm wPN}_{p'})_{L^\ast}$, and $({\rm wPNR}_{p})_{L}$
and the local property $({\rm LocN}_{p})_{L}$:
\begin{equation}\label{e1.14}
({\rm PN}_{p'})_{L^*}\Leftrightarrow({\rm PNR}_{p})_{L}\Rightarrow({\rm LocN}_{p})_{L}\Leftrightarrow({\rm wPN}_{p'})_{L^\ast}
\Leftrightarrow({\rm wPNR}_{p})_{L}.
\end{equation}
Assume further that $\alpha$ satisfies the assumption of Theorem \ref{StrongLocWeak}(ii).
Then Theorems \ref{PRandPRR} and \ref{StrongLocWeak} yield that, in the Robin case,
\begin{equation*}
({\rm PR}_{p'})_{L^*}\Leftrightarrow({\rm PRR}_{p})_{L}\Rightarrow({\rm LocR}_{p})_{L}
\Rightarrow({\rm wPR}_{p'})_{L^\ast}\Leftrightarrow({\rm wPRR}_{p})_{L},
\end{equation*}
which is weaker than the corresponding conclusion \eqref{e1.14}. Indeed, the (local) Neumann problem has the property that,
if $u$ is a weak solution to the (local) Neumann problem, then,
for any constant $C$, $u+C$ is still its  weak solution.
This property guarantees that Poincar\'e's inequality is always available to deal with the (local) Neumann problem.
Using this, Feneuil and Li \cite[Theorem 1.19]{fl24} proved the equivalence
between the problem $({\rm wPN}_{p'})_{L^\ast}$
and the local property $({\rm LocN}_{p})_{L}$. However, when dealing
with the (local) Robin problem, since Poincar\'e's inequality is
not always valid, we can not prove the equivalence between $({\rm wRN}_{p'})_{L^\ast}$
and $({\rm LocR}_{p})_{L}$.

\item[{\rm(iii)}] We prove Theorems \ref{PRandPRR},
\ref{StrongLocWeak}, and \ref{RandPR}
by applying some methods and techniques from Mourgoglou et al. \cite{mpt22}
and Feneuil and Li \cite{fl24} and also using
several properties of weak solutions to the Robin problem,
such as Caccioppoli's inequality, the Moser type estimate, and
the boundary H\"older regularity estimate for weak solutions.
Compared with \cite{mpt22} and \cite{fl24}, the main new ingredients
appearing in the proofs of the aforementioned theorems are
establishing the equivalence between the Robin problems $(\mathrm{R}_p)_{L}$ and $(\widetilde{\mathrm{R}}_{p})_{L}$
and finding the suitable variant of the local property $({\rm LocR}_{p})_{L}$.
\end{itemize}
\end{remark}

As a consequence of all the results above, we
obtain an analogue of the extrapolation theorem for the $L^p$
Neumann problem obtained by Feneuil and Li \cite{fl24} in the Robin case.
Indeed, as mentioned at the beginning, Feneuil and Li \cite{fl24} improved the
extrapolation theorem for the $L^p$ Neumann problem established by Kenig and Pipher \cite{kp93} and Hofmann and Sparrius
\cite{hs25} by showing that, for any $p\in(1,\infty)$, there exists a positive constant $\varepsilon$ such that, for any given
$q\in(1,p+\varepsilon)$,
\begin{equation}\label{e1.15}
({\rm D}_{p'})_{L^*}+({\rm N}_p)_{L}\Rightarrow({\rm N}_q)_L.
\end{equation}
As an immediate consequence of Theorems \ref{PRandPRR}, \ref{StrongLocWeak}, and \ref{RandPR}, we obtain the following conclusion.

\begin{corollary}\label{c1.1}
Let $n\ge2$, $\Omega\subset\mathbb{R}^n$ be a bounded one-sided {\rm CAD}, $p\in(1,\infty)$, $p'$ be the H\"older conjugate of $p$,
$L:=-\mathrm{div}(A\nabla\cdot)$ be a uniformly elliptic operator, and $\alpha$ satisfy \eqref{e1.3} and $\alpha\in L^p(\partial\Omega,\sigma)$ with $\|\alpha \|_{L^p(B(x,r)\cap\partial\Omega,\sigma)}
\le Cr^{(n-1)/p}$ for any $x\in\partial\Omega$ and $r\in(0,\mathrm{diam\,}(\Omega))$, where $C$ is a positive constant independent of $x$ and $r$.
Then there exists a positive constant $\varepsilon$, depending on $p$, $L$, and $\Omega$, such that, if both $({\rm D}_{p'})_{L^*}$ and
$({\rm R}_p)_{L}$ are solvable, then $({\rm R}_q)_{L}$ is solvable for any $q\in(1,p+\varepsilon)$.
\end{corollary}

Corollary \ref{c1.1} give the extrapolation theorem for the Robin problem
$({\rm R}_p)_{L}$ under the assumption that $({\rm D}_{p'})_{L^*}$ is solvable,
which can be seen as a natural analogue of \eqref{e1.15} in the Robin case.

The remainder of this article is organized as follows.

In Section \ref{s2}, we present several geometric assumptions on domains
considered in this article, and recall some properties of both weak solutions
to the Robin problem \eqref{e1.4} and
Green's function associated with the Robin problem.

In Section \ref{s3}, we give the proofs of Theorems \ref{PRandPRR} and \ref{StrongLocWeak}, while, in Section \ref{s4}, we first prove the
equivalence between the solvability of the Robin problems
$({\rm R}_p)_L$ and $({\rm \widetilde{R}}_p)_L$ (see Proposition \ref{RobinEquiv}) by applying the geometric assumption
of CAD, and then we show Theorem \ref{RandPR} by using the equivalence of
$({\rm R}_p)_L$ and $({\rm \widetilde{R}}_p)_L$ and a good-$\lambda$ argument.

Furthermore, in Section \ref{s5}, we apply Theorems \ref{PRandPRR}, \ref{StrongLocWeak}, and \ref{RandPR}
to obtain the optimal range of the solvability of the $L^p$ Poisson--Robin problem
and the $L^p$ Poisson--Robin-regularity problem for the Laplace operator on bounded Lipschitz domains or $C^1$ domains.

We end this introduction by making some conventions on symbols. Throughout this article,
we \emph{always} denote by $C$ a positive constant which is independent of the main parameters,
but it may vary from line to line. We also use the \emph{symbol} $C_{(\alpha,\beta,\ldots)}$ or
$c_{(\alpha,\beta,\ldots)}$ to denote a positive constant depending on the indicated parameters $\alpha,\beta,\ldots.$
The \emph{symbol} $f\lesssim g$ means that $f\le Cg$. If $f\lesssim g$
and $g\lesssim f$, then we write $f\sim g$. If $f\le Cg$ and $g=h$ or $g\le h$, we then write $f\lesssim g=h$
or $f\lesssim g\le h$. Assume that $\mathbb{N}:=\{1,2,\ldots\}$. For any $p\in[1,\infty]$, the \emph{symbol} $p'$ denotes its
\emph{H\"older conjugate}, that is, $1/p'+1/p=1$.
For any subset $E$ of $\mathbb{R}^n$, the \emph{symbol} $\mathbf{1}_E$ denotes its \emph{characteristic function}
and $E^\complement$ its \emph{complement} in $\mathbb{R}^n$. For any $x\in\mathbb{R}^n$ and any nonempty sets $E_1,E_2\subset
\mathbb{R}^n$, let
$$
d(x,E_1):=\inf\{|x-y|:y\in E_1\},\  \ d(E_1,E_2):=\inf\{|y-z|: y\in E_1,z\in E_2\},$$
and
$$\mathrm{diam\,}(E_1):=\sup\{|y-z|: y,z\in E_1\}.
$$
For any given domain $\Omega\subset\mathbb{R}^n$ and any $x\in\Omega$, let $\delta(x):=\mathrm{dist\,}(x,\partial\Omega)$.
For any $x\in \mathbb{R}^{n}$ and $r\in(0,\infty)$, we use $B$ or $B(x,r)$
to denote the ball centered at $x$ with radius $r$, that is,
$B:=B(x,r):=\{y\in\mathbb{R}^{n}: |x-y|<r\};$ also, for any $\alpha\in(0,\infty)$,
$\alpha B$ denotes the ball with the same center as $B$, but with
$\alpha$ times radius of $B$. For any measurable set $E\subset\mathbb{R}^n$ and $f\in L^1(E)$, we denote the integral
$\int_{E}|f(x)|\,dx$ simply by $\int_{E}|f|\,dx$ and, when $|E|\in(0,\infty)$, we always use the following symbol
$$
\fint_Ef\,dx:=\frac{1}{|E|}\int_{E}f(x)\,dx.
$$
Finally, in all proofs we
consistently retain the symbols
introduced in the original theorem (or related statement).

\section{Preliminaries\label{s2}}

In this section, we first introduce the geometric assumptions on domains considered in this article, and then recall some properties
of weak solutions to the Robin problem \eqref{e1.4} and
Green's function associated with the Robin problem.

\begin{definition}\label{CSandHC}
Let $n\ge2$ and $\Omega\subset\mathbb{R}^n$ be a domain.
\begin{enumerate}[{\rm (a)}]
\item The domain $\Omega$ is said to satisfy the \emph{$c$-corkscrew condition}
for some $c\in(0,1)$ if, for any $x\in\partial\Omega$
and $r\in(0,\mathrm{diam\,}(\Omega))$, there exists $y_x\in B(x,r)\cap\Omega$ such that $B(y_x,cr)\subset B(x,r)\cap\Omega$.
The point $y_x\in\Omega$ is called a \emph{corkscrew point} related to $B(x,r)$.

\item Let $N\in\mathbb{N}$ and $M\in[1,\infty)$. Then $x,y\in\Omega$ are
said to be linked by an \emph{$(M,N)$-Harnack chain}
if there exists a chain of open balls $B_1,\ldots,B_N\subset\Omega$ such that
$x\in B_1$, $y\in B_N$, $B_k\cap B_{k+1}\neq\emptyset$ for any $k\in\{1,\ldots,N-1\}$, and $M^{-1}\mathrm{diam\,}(B_k)\le\mathrm{dist\,}(B_k,\partial\Omega)\le M\mathrm{diam\,}(B_k)$
for any $k\in\{1,\ldots,N\}$.

\item The domain $\Omega$ is said to satisfy the \emph{Harnack chain
condition} if there exists a uniform constant $M\in[1,\infty)$
such that, for any $x,y\in\Omega$, there exists an $(M,N)$-Harnack chain connecting $x$ and $y$, with $N\in\mathbb{N}$
depending only on $M$ and on the ratio $|x-y|/\min\{\delta(x),\delta(y)\}$.
\end{enumerate}
\end{definition}

It is well known that the $L^p$ norms of the modified non-tangential maximal function, the Lusin area function, and the averaged
Carleson functional are mutually equivalent under the change of the aperture $a$ or the radius $c$ (see, for example, \cite{cms85,fl24,mpt22}),
and the equivalence between the area function $\widetilde{\mathcal{A}}_1$ and the Carleson functional $\widetilde{\mathcal{C}}_1$
can be found in \cite[Theorem 3]{cms85}.

\begin{lemma}\label{TentEqui}
Let $n\ge2$, $\Omega\subset\mathbb{R}^n$ be a bounded one-sided {\rm CAD}, $a\in(1,\infty)$, $c\in(0,1)$, and $p\in(0,\infty]$. Then, for any $u\in L_{\rm c}^{\infty}(\Omega)$,
$$
\left\|\widetilde{\mathcal{N}}^{a,c}(u)\right\|_{L^p(\partial\Omega,\sigma)}\sim
\left\|\widetilde{\mathcal{N}}(u)\right\|_{L^p(\partial\Omega,\sigma)},
$$
$$
\left\|\widetilde{\mathcal{A}}_1^{a,c}(u)\right\|_{L^p(\partial\Omega,\sigma)}\sim
\left\|\widetilde{\mathcal{A}}_1(u)\right\|_{L^p(\partial\Omega,\sigma)},
$$
and
$$
\left\|\widetilde{\mathcal{C}}_1^{c}(u)\right\|_{L^p(\partial\Omega,\sigma)}\sim
\left\|\widetilde{\mathcal{C}}_1(u)\right\|_{L^p(\partial\Omega,\sigma)},
$$
where the positive equivalence constants depend only on $\Omega$, $a$, $c$, and $p$. Moreover, if $p\in(1,\infty)$,
then, for any $u\in L_{\rm c}^{\infty}(\Omega)$,
$$
\left\|\widetilde{\mathcal{A}}_1(u)\right\|_{L^p(\partial\Omega,\sigma)}\sim
\left\|\widetilde{\mathcal{C}}_1(u)\right\|_{L^p(\partial\Omega,\sigma)},
$$
where the positive equivalence constants depend only on $\Omega$ and $p$.
\end{lemma}

Let $p\in[1,\infty]$. Based on the conclusions in Lemma \ref{TentEqui}, we define \emph{tent spaces} $\widetilde{T}^p_{\infty}(\Omega)$
and $\widetilde{T}^{p}_1(\Omega)$, respectively, by setting
$$
\widetilde{T}^p_{\infty}(\Omega):=\left\{u\in L^2_{\mathrm{loc}}(\Omega): \|u\|_{\widetilde{T}^p_{\infty}(\Omega)}:=\left\|\widetilde{\mathcal{N}}
(u)\right\|_{L^p(\partial\Omega,\sigma)}<\infty\right\}
$$
and
$$
\widetilde{T}^p_{1}(\Omega):=\left\{u\in L^2_{\mathrm{loc}}(\Omega):\|u\|_{\widetilde{T}^p_{1}(\Omega)}:
=\left\|\widetilde{\mathcal{C}}_1(u)\right\|_{L^p(\partial\Omega,\sigma)}<\infty\right\}.
$$

Now, we present the duality between the modified non-tangential maximal function $\widetilde{\mathcal{N}}$ and the averaged Carleson
functional $\widetilde{\mathcal{C}}_1$ and a conclusion on the density of the set $L^\infty_{\rm c}(\Omega)$ in tent spaces $\widetilde{T}_1^p(\Omega)$
(see, for example, \cite{fl24,mpt22}).

\begin{lemma}\label{NdualC}
Let $n\ge2$, $\Omega\subset\mathbb{R}^n$ be a bounded one-sided {\rm CAD}, $p\in(1,\infty)$, and $p'$ be the H\"older conjugate of $p$.
Then, for any $u\in\widetilde{T}^p_{\infty}(\Omega)$ and $F\in\widetilde{T}^{p'}_1(\Omega)$,
\begin{equation}\label{e2.1}
\left|\int_{\Omega}uF\,dx\right|\lesssim\left\|\widetilde{\mathcal{N}}(u)
\right\|_{L^p(\partial\Omega,\sigma)}
\left\|\widetilde{\mathcal{C}}_1(\delta F)\right\|_{L^{p'}(\partial\Omega,\sigma)},
\end{equation}
with the implicit positive constant depending only on $p$ and $\Omega$.
Moreover, for any $u\in\widetilde{T}^p_{\infty}(\Omega)$,
\begin{equation}\label{e2.2}
\left\|\widetilde{\mathcal{N}}(u)\right\|_{L^p(\partial\Omega,\sigma)}
\lesssim\sup_{F:\,\|\widetilde{\mathcal{C}}_1(\delta F)\|_{L^{p'}(\partial\Omega,\sigma)}=1}\left|\int_{\Omega}uF\,dx\right|
\end{equation}
and, for any $u\in\widetilde{T}^{p'}_1(\Omega)$,
\begin{equation*}
\left\|\widetilde{\mathcal{C}}_1(u)\right\|_{L^{p'}(\partial\Omega,\sigma)}
\lesssim\sup_{F:\,\|\widetilde{\mathcal{N}}(\delta F)\|_{L^p(\partial\Omega,\sigma)}=1}\left|\int_{\Omega}uF\,dx\right|,
\end{equation*}
where the implicit positive constants depend only on $p$ and $\Omega$.
\end{lemma}

\begin{lemma}\label{Dense}
Let $n\ge2$, $\Omega\subset\mathbb{R}^n$ be a bounded one-sided {\rm CAD}, and $p\in[1,\infty)$. Then $L_{\rm c}^{\infty}(\Omega)$
is dense in $\widetilde{T}_1^{p}(\Omega)$.
\end{lemma}

Moreover, we also need Caccioppoli's inequality, the Moser type estimate, and the boundary H\"older regularity estimate
for weak solutions to the Robin problem \eqref{e1.4}. The following Lemma \ref{Cacciopoli} presents Caccioppoli's inequality
for weak solutions.

\begin{lemma}\label{Cacciopoli}
Let $n\ge2$, $\Omega\subset\mathbb{R}^n$ be a bounded one-sided {\rm CAD}, $L:=-\mathrm{div}(A\nabla\cdot)$ be a uniformly elliptic operator,
$\boldsymbol{F}\in L_{\rm c}^{\infty}(\Omega;\mathbb{R}^n)$, and $\alpha$ be the same as in \eqref{e1.3}. Then there exists a positive constant $C$,
depending only on $\Omega$, $n$, and $\mu_0$, such that
\begin{itemize}
\item[{\rm(i)}] for any ball $B:=B(x,r)$ satisfying $2B \subset \Omega$
and for any weak solution $u\in W^{1,2}(2B)$ to $Lu=-\mathrm{div\,}\boldsymbol{F}$ in $2B$,
$$
\left(\fint_B |\nabla u|^2 \,dx\right)^{\frac{1}{2}}\le C\left[\frac{1}{r}\left(\fint_{2B} |u|^2\,dx\right)^{\frac{1}{2}}+
\left(\fint_{2B}|\boldsymbol{F}|^2\,dx\right)^{\frac{1}{2}}\right];
$$
\item[{\rm(ii)}] for any ball $B:=B(x,r)$ with $x\in\partial\Omega$ and $r\in(0,\mathrm{diam\,}(\Omega))$ and for any
weak local solution $u\in W^{1,2}(2B\cap \Omega)$ to $Lu=-\mathrm{div\,}\boldsymbol{F}$ with the Robin boundary condition
$\frac{\partial u}{\partial\boldsymbol{\nu}}+\alpha  u=\boldsymbol{F}\cdot\boldsymbol{\nu}$ on $2B\cap\partial\Omega$,
$$
\left(\fint_{B\cap \Omega} |\nabla u|^2\,dx\right)^{\frac{1}{2}}\le C\left[\frac{1}{r} \left(\fint_{2B \cap \Omega} |u|^2 \, dx\right)^{\frac{1}{2}}+\left(\fint_{2B\cap\Omega}|\boldsymbol{F}|^2\,dx\right)^{\frac{1}{2}}\right].
$$
\end{itemize}
\end{lemma}

\begin{proof}
Here we only give the proof of (ii) because the proof of (i) is similar.

Let $B:=B(x,r)$ with $x\in\partial\Omega$ and $r\in(0,\mathrm{diam\,}(\Omega))$ and let $u\in W^{1,2}(2B\cap \Omega)$ be a weak solution
to $Lu=-\mathrm{div\,}\boldsymbol{F}$ in $2B\cap \Omega$ with the Robin boundary condition
$\frac{\partial u}{\partial\boldsymbol{\nu}}+\alpha  u=\boldsymbol{F}\cdot\boldsymbol{\nu}$ on $2B\cap\partial\Omega$.
Then, for any $\phi\in W^{1,2}(2B\cap\Omega)$,
$$
\int_{\Omega}A\nabla u\cdot \nabla\phi\,dx+\int_{\partial\Omega}\alpha  u\phi\,d\sigma=\int_{\Omega}\boldsymbol{F}\cdot\nabla\phi\,dx.
$$
Take $\psi\in C_{\rm c}^{\infty}(2B)$ be such that $0\le\psi\le1$, $\psi\equiv1$ on $B$, and $|\nabla\psi|\lesssim\frac{1}{r}$.
Let $\phi:=u\psi^2$.
Then
\begin{align}\label{e2.3}
\int_{\Omega}|\nabla u|^2\psi^2\,dx
&\lesssim\int_{\Omega} A\nabla u\cdot\nabla u\psi^2\,dx+\int_{\partial\Omega}\alpha  u^2\psi^2\,d\sigma\nonumber\\ \nonumber
&=\int_{\Omega}A\nabla u\cdot\nabla(u\psi^2)\,dx-2\int_{\Omega}u\psi A\nabla u\cdot\nabla\psi \,dx+\int_{\partial\Omega}\alpha  u^2\psi^2\,d\sigma\\
&=\int_{\Omega}\boldsymbol{F}\cdot\nabla(u\psi^2)\,dx-2\int_{\Omega}u\psi A\nabla u\cdot\nabla\psi\,dx:=\mathrm{I}_1+\mathrm{I}_2.
\end{align}
For the term $\mathrm{I}_1$, from Young's inequality, we deduce that, for any given $\varepsilon\in(0,1/2)$,
\begin{align}\label{e2.4}
|\mathrm{I}_1|
&\le\left|\int_{\Omega}\psi^2\boldsymbol{F}\cdot\nabla u\,dx\right|+2\left|\int_{\Omega}u\psi\boldsymbol{F}\cdot\nabla\psi\,dx\right|\nonumber\\
&\le\varepsilon\int_{\Omega}|\nabla u|^2\psi^2\,dx+C_{(\varepsilon)}\int_{\Omega}|\boldsymbol{F}|^2\psi^2\,dx
+\int_{\Omega}|\boldsymbol{F}|^2\psi^2\,dx+\int_{\Omega}|u|^2|\nabla\psi|^2\,dx.
\end{align}
For $\mathrm{I}_2$, similar to the estimation of \eqref{e2.4}, we have
\begin{align}\label{e2.5}
|\mathrm{I}_2|
\le\varepsilon\int_{\Omega}|\nabla u|^2\psi^2\,dx+C_{(\varepsilon)}\int_{\Omega}|u|^2|\nabla\psi|^2\,dx.
\end{align}
Letting $\varepsilon:=1/4$ in \eqref{e2.4} and \eqref{e2.5} and using \eqref{e2.3}, \eqref{e2.4}, and \eqref{e2.5}, we conclude that
\begin{align*}
\int_{B\cap\Omega}|\nabla u|^2\,dx&\le\int_{\Omega}|\nabla u|^2\psi^2\,dx\lesssim\int_{\Omega}|u|^2|\nabla\psi|^2\,dx
+\int_{\Omega}|\boldsymbol{F}|^2\psi^2\,dx\\
&\lesssim\int_{2B\cap\Omega}|\boldsymbol{F}|^2\,dx+\frac{1}{r^2}\int_{2B\cap\Omega}|u|^2\,dx,
\end{align*}
which further implies that
$$
\left(\fint_{B\cap\Omega}|\nabla u|^2\,dx\right)^{\frac{1}{2}}
\lesssim\frac{1}{r}\left(\fint_{2B\cap\Omega}|u|^2\,dx\right)^{\frac{1}{2}}
+\left(\fint_{2B\cap\Omega}|\boldsymbol{F}|^2\,dx\right)^{\frac{1}{2}}.
$$
This finishes the proof of (ii) and hence Lemma \ref{Cacciopoli}.
\end{proof}

\begin{lemma}\label{Moser}
Let $n\ge2$, $\Omega\subset\mathbb{R}^n$ be a bounded one-sided {\rm CAD}, $L:=-\mathrm{div}(A\nabla\cdot)$ be a uniformly elliptic operator,
and $\alpha$ be the same as in \eqref{e1.3}. Then there exists a positive constant $C$, depending only on $\Omega$, $n$, and $L$, such that
\begin{itemize}
\item[{\rm(i)}] for any ball $B:=B(x,r)$ satisfying that $2B \subset \Omega$
and for any weak solution $u\in W^{1,2}(2B)$ to $Lu=0$ in $2B$,
$$
\underset{x\in B}{\sup}\, |u(x)|\le C \fint_{2B} |u|\,dx;
$$
\item[{\rm(ii)}]  for any ball $B:=B(x,r)$ with $x\in\partial\Omega$ and $r\in(0,\mathrm{diam\,}(\Omega))$
and for any weak local solution $u\in W^{1,2}(2B\cap \Omega)$
to $Lu=0$ with the Robin boundary condition $\frac{\partial u}{\partial\boldsymbol{\nu}}+\alpha  u=0$ on $2B\cap\partial\Omega$,
$$
\underset{x\in B\cap \Omega}{\sup} |u(x)| \le C \fint_{2B \cap \Omega}|u|\,dx.
$$
\end{itemize}
\end{lemma}

Lemma \ref{Moser} gives the Moser type estimate for weak solutions to the Robin problem. Lemma \ref{Moser}(i) is classical
(see, for instance, \cite[Lemma 1.1.8]{k94}), and Lemma \ref{Moser}(ii) was essentially obtained in \cite[Lemma 3.1]{ls05}.

Moreover, as a corollary of Caccioppoli's inequality as in Lemma \ref{Cacciopoli} and the Moser type estimate as in Lemma \ref{Moser},
we have the following Meyers type estimate for the weak solution to the local Robin problem.

\begin{corollary}\label{c2.1}
Let $n\ge2$, $\Omega\subset\mathbb{R}^n$ be a bounded one-sided {\rm CAD}, $L:=-\mathrm{div}(A\nabla\cdot)$ be a uniformly elliptic operator,
and $\alpha$ be the same as in \eqref{e1.3}. Assume that $B:=B(x,r)$ with $x\in\partial\Omega$ and $r\in(0,\mathrm{diam\,}(\Omega))$
and $u\in W^{1,2}(4B\cap \Omega)$ is a weak solution to $Lu=0$ with the Robin boundary condition
$\frac{\partial u}{\partial\boldsymbol{\nu}}+\alpha  u=0$ on $4B\cap\partial\Omega$. Then there exists a positive constant $\varepsilon$,
depending only on $\Omega$ and $L$, such that
\begin{equation}\label{e2.6}
\left(\fint_{B\cap \Omega} |\nabla u|^{2+\varepsilon}\,dx+\fint_{B\cap \Omega} |u|^{2+\varepsilon}\,dx\right)^{\frac{1}{2+\varepsilon}}
\le C\left(\fint_{4B \cap \Omega} |\nabla u|\, dx
+\fint_{4B \cap \Omega} |u|\, dx\right),
\end{equation}
where $C$ is a positive constant independent of both $B$ and $u$.
\end{corollary}
\begin{proof}
Applying Lemma \ref{Moser}(ii) and an argument similar to that used in the proof of Lemma \ref{Cacciopoli}(ii), we obtain that
\begin{equation}\label{e2.7}
\left(\fint_{B\cap\Omega}|\nabla u|^2\,dx\right)^{\frac{1}{2}}
\lesssim\frac{1}{r}\left(\fint_{2B\cap\Omega}|u-c_u|^2\,dx\right)^{\frac{1}{2}}+\fint_{2B\cap\Omega}|u|\,dx,
\end{equation}
where $c_u:=\fint_{2B\cap\Omega} u\,dx$. Moreover, by Poincar\'e's inequality on $2B\cap\Omega$ (see, for instance, \cite[Lemma 2.2]{ddemm24}),
we find that there exists $p_0\in[1,2)$ such that
\begin{equation*}
\left(\fint_{2B\cap\Omega}|u-c_u|^2\,dx\right)^{\frac{1}{2}}\lesssim r\left(\fint_{3B\cap\Omega}|\nabla u|^{p_0}\,dx\right)^{\frac{1}{p_0}},
\end{equation*}
which, together with \eqref{e2.7} and Lemma \ref{Moser}(ii), implies that
\begin{equation}\label{e2.8}
\left(\fint_{B\cap\Omega}|\nabla u|^2\,dx+\fint_{B\cap\Omega}| u|^2\,dx\right)^{\frac{1}{2}}
\lesssim\left(\fint_{3B\cap\Omega}|\nabla u|^{p_0}\,dx+\fint_{3B\cap\Omega}| u|^{p_0}\,dx\right)^{\frac{1}{p_0}}.
\end{equation}
Then, applying \eqref{e2.8} and an argument similar to that
used in the proof of \cite[Lemma 4.38]{bm16}, we further conclude that
the estimate \eqref{e2.6} holds. This finishes the proof of Corollary \ref{c2.1}.
\end{proof}

\begin{lemma}\label{Holder}
Let $n\ge2$, $\Omega\subset\mathbb{R}^n$ be a bounded one-sided {\rm CAD}, $L:=-\mathrm{div}(A\nabla\cdot)$ be a uniformly elliptic operator,
and $\alpha$ be the same as in \eqref{e1.3}. Then there exist positive constants $C\in(0,\infty)$ and $\kappa\in(0,1]$,
depending only on $\Omega$ and $L$, such that
\begin{itemize}
\item[{\rm(i)}] for any ball $B:=B(x,r) \subset \Omega$, any weak solution $u\in W^{1,2}(B)$ to $Lu=0$ in $B$, and any $\epsilon \in (0,1)$,
$$
\underset{\epsilon B}{\mathrm{osc}}\, u\le C \epsilon^\kappa \underset{B}{\mathrm{osc}}\, u,
$$
where, for any given measurable set $E\subset\mathbb{R}^n$, $\underset{E}{\mathrm{osc}}\, u:=\underset{E}{\sup}\,u-\underset{E}{\inf}\,u$;
\item[{\rm(ii)}]  for any ball $B:=B(x,r)$ with $x\in\partial\Omega$ and $r\in(0,\mathrm{diam\,}(\Omega))$, any weak solution
$u\in W^{1,2}(2B\cap \Omega)$ to $Lu=0$ with the Robin boundary condition $\frac{\partial u}{\partial\boldsymbol{\nu}}+\alpha  u=0$ on
$2B\cap\partial\Omega$, and any $\varepsilon\in(0,1)$,
$$
\underset{\epsilon B\cap \Omega}{\mathrm{osc}}\,u \le C \epsilon^\kappa \underset{B\cap \Omega}{\mathrm{osc}}\, u.
$$
\end{itemize}
\end{lemma}

Lemma \ref{Holder}, essentially established in \cite[Lemma 4.1]{V12} and \cite[Theorem 4.5]{ddemm24},
presents the H\"older regularity estimate for weak solutions to the Robin problem.

Next, we recall some properties of Green's function for the Robin problem.

\begin{lemma}\label{Green}
Let $n\ge2$, $\Omega\subset\mathbb{R}^n$ be a bounded one-sided {\rm CAD}, $L:=-\mathrm{div}(A\nabla\cdot)$ be a uniformly elliptic operator,
and $\alpha$ be the same as in \eqref{e1.3}. Then there exists a non-negative function $G_R$ defined on $\Omega\times\Omega$,
called Green's function for the Robin problem, such that the following properties hold.
\begin{enumerate}[{\rm (i)}]
\item For any $x\in\Omega$ and $r\in(0,\mathrm{diam\,}(\Omega))$,
$$
G_R(\cdot,x)\in W^{1,2}(\Omega\setminus B(x,r))\cap W^{1,1}(\Omega).
$$
\item For any $x\in\Omega$ and any $\psi\in C^{\infty}(\Omega)$,
\begin{equation}\label{E2.9}
\int_{\Omega}A(y)\nabla G_R(y,x)\cdot\nabla\psi(y)\,dy+\int_{\partial\Omega}\alpha (y) G_R(y,x)\psi(y)\,d\sigma(y)=\psi(x).
\end{equation}
\item For any $f\in C_{\rm c}(\partial\Omega)$, the solution $u$ to the Robin problem \eqref{e1.6} can be represented as,
for any $x\in\Omega$,
\begin{equation}\label{e2.10}
u(x)=\int_{\partial\Omega}G_R(y,x)f(y)\,d\sigma(y);
\end{equation}
Moreover, for any $h\in L_{\rm c}^{\infty}(\Omega)$ and $\boldsymbol{F}\in L_{\rm c}^{\infty}(\Omega;\mathbb{R}^n)$, the solution $u$ to the
Robin problem \eqref{e1.8} can be represented by, for any $x\in\Omega$,
\begin{equation}\label{e2.11}
u(x)=\int_{\Omega}\left[G_R(x,y)h(y)+\nabla_{y}G_R(x,y)\cdot\boldsymbol{F}(y)\right]\,dy.
\end{equation}

\item There exists a positive constant $C$, depends only on $\Omega$ and $L$, such that, for any $x,y\in\Omega$ with $x\neq y$,
\begin{equation}\label{e2.12}
0 \le G(x,y) \le \begin{cases}
C\left[ 1 + \log\left( \dfrac{1}{|x-y|} \right) \right],&\text{if }n=2,\\
\dfrac{C}{|x-y|^{n-2}},&\text{if }n\ge 3.
\end{cases}
\end{equation}

\item When $n\ge3$, there exists a positive constant $C$, depends only
on $\Omega$, $n$, and $L$, such that, for any ball $B:=B(z,r)$,
with $z\in\partial\Omega$ and $r\in(0,\mathrm{diam\,}(\Omega))$, and for any $y\in\Omega\setminus \frac{3}{2}B$,
$$
\left[\fint_{B\cap\Omega}|G_R(x,y)|^2\,dx\right]^{\frac12}\le Cr^{2-n}\ \ \text{and}\ \ \left[\fint_{B\cap\Omega}
\left|\nabla_xG_R(x,y)\right|\,dx\right]^{\frac12}\le Cr^{1-n}.
$$
\end{enumerate}
\end{lemma}

\begin{proof}
The existence of Green's function $G_R$, (i), (ii), and \eqref{e2.10}
was obtained in \cite[Theorem 5.6]{ddemm24} and \cite[Theorem 1.5]{wy25}.
Moreover, the formula \eqref{e2.11} is a direct consequence of \eqref{E2.9}
and \eqref{e1.5}. Thus, (iii) holds.

Furthermore, (iv) of Lemma \ref{Green} was essentially obtained in \cite[Theorem 4.1]{CK14} (see also \cite[Theorem 1.5]{wy25}).
Using the pointwise upper bound estimate \eqref{e2.12} and Lemmas \ref{Cacciopoli} and \ref{Moser}, we further conclude that the conclusion
of (v) holds. This finishes the proof of Lemma \ref{Green}.
\end{proof}

\begin{remark}\label{r2.1}
Let $L$ and $G_R$ be as in Lemma \ref{Green}. Let $L^\ast:=-\mathrm{div}(A^T\nabla\cdot)$ be the adjoint operator of $L$ and
denote by $G^T_R$ Green's function for the Robin problem
associated with $L^\ast$. Then it is easy to find that,
for any $x,y\in\Omega$, $G^T_R(x,y)=G_R(y,x)$.
\end{remark}

Moreover, we also need the following Poincar\'e's inequality
(see, for instance, \cite[Lemma 2.2]{ddemm24}).

\begin{lemma}\label{l2.2}
Let $n\ge2$ and $\Omega\subset\mathbb{R}^n$ be a bounded one-sided {\rm CAD}.
Then, for any $x\in\partial\Omega$, $r\in(0,\mathrm{dist\,}(\Omega))$,
and $u\in W^{1,2}(\Omega)$,
$$
\left[\fint_{B(x,r)\cap\Omega}|u-\overline{u}_E|^2\,dy\right]^{\frac12}
\le Cr\left[\fint_{B(x,2r)\cap\Omega}|\nabla u|^2\,dy\right]^{\frac12},
$$
where $\overline{u}_E:=\fint_E u\,dy$ with $E\subset B(x,2r)\cap\Omega$
satisfying $|E|\ge c|B(x,r)\cap\Omega|$ for some $c\in(0,1]$ and $C$ is a positive constant depending only on $c$ and the geometric constants of $\Omega$.
\end{lemma}

Let $n\ge2$ and $\Omega\subset\mathbb{R}^n$ be a one-sided {\rm CAD}. Assume that $f\in L^1_{\mathrm{loc}}(\partial\Omega,\sigma)$.
The \emph{Hardy--Littlewood maximal function} $\mathcal{M}(f)$ of $f$ is defined by setting, for any $x\in\partial\Omega$,
$$
\mathcal{M}(f)(x) :=\sup_{B\ni x}\frac{1}{\sigma(B\cap\partial\Omega)}\int_{B\cap\partial\Omega}|f|\,d\sigma,
$$
where the supremum is taken over all balls $B$, centered at the boundary $\partial\Omega$, containing $x\in\partial\Omega$.

Then we have the following boundedness of the Hardy--Littlewood maximal function $\mathcal{M}$ (see, for example, \cite{s93}).
\begin{lemma}\label{l2.1}
Let $n\ge2$, $\Omega\subset\mathbb{R}^n$ be a one-sided {\rm CAD}, and $p\in(1,\infty)$.
Then the Hardy--Littlewood maximal operator $\mathcal{M}$ is bounded on $L^p(\partial\Omega,\sigma)$
and bounded from $L^1(\partial\Omega,\sigma)$ to $WL^{1}(\partial\Omega,\sigma)$.
Here the weak Lebesgue space $WL^{1}(\partial\Omega,\sigma)$
is defined by setting
$$WL^1(\Omega,\sigma):=\left\{f\ \text{is measurable on}\ \partial\Omega:\ \|f\|_{WL^1(\partial\Omega,\sigma)}<\infty\right\},
$$
where, for any measurable function $f$ on $\partial\Omega$,
$$\|f\|_{WL^1(\partial\Omega,\sigma)}:=\sup_{\lambda\in(0,\infty)}
\lambda\sigma\left(\left\{x\in\partial\Omega:\ |f(x)|>\lambda\right\}\right).$$
\end{lemma}

\section{Proofs of Theorems \ref{PRandPRR} and \ref{StrongLocWeak}\label{s3}}

In this section, we give the proofs of Theorems \ref{PRandPRR} and \ref{StrongLocWeak}. We first show the equivalence
between the (weak) Poisson--Robin problem and the (weak) Poisson--Robin-regularity problem
by using the duality between the modified non-tangential maximal function $\widetilde{\mathcal{N}}$ and the averaged Carleson functional $\widetilde{\mathcal{C}}_1$.

\begin{proof}[Proof of Theorem \ref{PRandPRR}]
We first prove (i). We only need to prove $(a)\Rightarrow(b)$ and $(c)\Rightarrow(a)$, because $(b)\Rightarrow(c)$ is trivial.

We now show $(c)\Rightarrow(a)$. Let $\boldsymbol{F}\in L_{\rm c}^{\infty}(\Omega;\mathbb{R}^n)$ and $u$ be the weak solution to the Robin problem
\begin{equation}\label{e3.1}
\begin{cases}
L^*u=-\mathrm{div\,}\boldsymbol{F}\ \ & \text{in}\ \ \Omega,\\
\displaystyle\frac{\partial u}{\partial\boldsymbol{\nu}}+\alpha  u=\boldsymbol{F}\cdot\boldsymbol{\nu} \ \ & \text{on}\ \ \partial\Omega.
\end{cases}
\end{equation}
To prove $(c)\Rightarrow(a)$, it suffices to show
\begin{equation}\label{e3.2}
\left\|\widetilde{\mathcal{N}}(u)\right\|_{L^{p'}(\partial\Omega,\sigma)}\lesssim
\left\|\widetilde{\mathcal{C}}_1(\delta|\boldsymbol{F}|)\right\|_{L^{p'}(\partial\Omega,\sigma)}.
\end{equation}
Let $E$ be a compact set in $\Omega$. Then $\|\widetilde{\mathcal{N}}(u\mathbf{1}_E)\|_{L^{p'}(\partial\Omega,\sigma)}<\infty$.
By the duality \eqref{e2.2} and the density of $L^\infty_{\rm c}(\Omega)$ in
$\widetilde{T}^p_1(\Omega)$ (see Lemma \ref{Dense}), we conclude that
there exists $h:=h_E\in L_{\rm c}^{\infty}(\Omega)$ such that
\begin{equation}\label{e3.3}
\left\|\widetilde{\mathcal{C}}_1(\delta h)\right\|_{L^{p}(\partial\Omega,\sigma)}\lesssim1
\end{equation}
and
\begin{equation}\label{e3.4}
\left\|\widetilde{\mathcal{N}}\left(u\mathbf{1}_E\right)\right\|_{L^{p'}
(\partial \Omega,\sigma)}\lesssim \int_{\Omega}uh\,dx.
\end{equation}
Let $v$ be the weak solution to the Robin problem
\begin{equation*}
\begin{cases}
Lv=h\ \ & \text{in}\ \ \Omega,\\
\displaystyle\frac{\partial v}{\partial\boldsymbol{\nu}}+\alpha  v=0 \ \ & \text{on}\ \ \partial\Omega.
\end{cases}
\end{equation*}
Then we have
\begin{align*}
\int_{\Omega}h u\,dx
&=\int_{\partial\Omega}\alpha  v u\,d\sigma+\int_{\Omega}A\nabla v\cdot\nabla u\,dx\\ \nonumber
&=\int_{\partial\Omega}\alpha  u v\,d\sigma+\int_{\Omega} A^T\nabla u\cdot\nabla v\,dx=\int_{\Omega}\boldsymbol{F}\cdot\nabla v\,dx,
\end{align*}
which, combined with \eqref{e2.1}, the solvability of $({\rm PRR}_p)_L$,  \eqref{e3.3}, and \eqref{e3.4}, further yields that
\begin{align}\label{e3.5}
\left\|\widetilde{\mathcal{N}}\left(u\mathbf{1}_E\right)\right\|_{L^{p'}
(\partial \Omega,\sigma)}
&\lesssim\left\|\widetilde{\mathcal{C}}_1(\delta\boldsymbol{F})\right\|_{L^{p'}(\partial\Omega,\sigma)}
\left\|\widetilde{\mathcal{N}}(\nabla v)\right\|_{L^{p}(\partial\Omega,\sigma)}\nonumber\\ \nonumber
&\lesssim\left\|\widetilde{\mathcal{C}}_1(\delta\boldsymbol{F})\right\|_{L^{p'}(\partial\Omega,\sigma)}
\left\|\widetilde{\mathcal{C}}_1(\delta h)\right\|_{L^{p}(\partial\Omega,\sigma)}\\
&\lesssim\left\|\widetilde{\mathcal{C}}_1(\delta\boldsymbol{F})\right\|_{L^{p'}(\partial\Omega,\sigma)}.
\end{align}
Since the right-hand side of the last inequality in \eqref{e3.5}
is independent of $E$, by taking $E\uparrow\Omega$, it follows from
\eqref{e3.5} that \eqref{e3.2} holds.
Therefore, $({\rm PR}_{p'})_{L^*}$ is solvable and hence
the conclusion of (a) holds.

Next, we prove $(a)\Rightarrow(b)$. Let $h\in L_{\rm c}^{\infty}(\Omega)$,
$\boldsymbol{F}\in L_{\rm c}^{\infty}(\Omega;\mathbb{R}^n)$, and $u$ be the weak solution to
the Robin problem
\begin{equation*}
\begin{cases}
Lu=h-\mathrm{div\,}\boldsymbol{F}\ \ & \text{in}\ \ \Omega,\\
\displaystyle\frac{\partial u}{\partial\boldsymbol{\nu}}+\alpha  u=f+\boldsymbol{F}\cdot\boldsymbol{\nu} \ \ & \text{on}\ \ \partial\Omega.
\end{cases}
\end{equation*}
To prove $(a)\Rightarrow(b)$, it suffices to show
\begin{equation}\label{e3.6}
\left\|\widetilde{\mathcal{N}}(\nabla u)\right\|_{L^{p}(\partial \Omega,\sigma)}
\lesssim\left\|\widetilde{\mathcal{C}}_1(\delta h)\right\|_{L^{p}(\partial\Omega,\sigma)}+
\left\|\widetilde{\mathcal{C}}_1(\boldsymbol{F})\right\|_{L^{p}(\partial\Omega,\sigma)}.
\end{equation}
Let $E$ be a compact set in $\Omega$.
Then $\|\widetilde{\mathcal{N}}(\nabla u\mathbf{1}_E)\|_{L^{p}(\partial\Omega,\sigma)}<\infty$
and there exists $\boldsymbol{G}:=\boldsymbol{G}_E\in L_{\rm c}^{\infty}(\Omega;\mathbb{R}^n)$
such that
\begin{equation}\label{e3.7}
\left\|\widetilde{\mathcal{C}}_1(\delta\boldsymbol{G})\right\|_{L^{p'}(\partial\Omega,\sigma)}\lesssim1
\end{equation}
and
\begin{equation}\label{e3.8}
\left\|\widetilde{\mathcal{N}}(\nabla u\mathbf{1}_E)\right\|_{L^{p}(\partial \Omega,\sigma)}
\lesssim \int_{\Omega}\nabla u\cdot\boldsymbol{G}\,dx.
\end{equation}
Let $v$ be the weak solution to the Robin problem \eqref{e3.1} with
$\boldsymbol{F}$ replaced by $\boldsymbol{G}$.
Then we have
\begin{align*}
\int_{\Omega}\nabla u\cdot\boldsymbol{G}\,dx
&=\int_{\partial\Omega}\alpha  v u\,d\sigma+\int_{\Omega}A^T\nabla v\cdot\nabla u\,dx\\ \nonumber
&=\int_{\partial\Omega}\alpha  v u\,d\sigma+\int_{\Omega} A\nabla u\cdot\nabla v\,dx
=\int_{\Omega}hv+\boldsymbol{F}\cdot\nabla v\,dx.
\end{align*}
From this, \eqref{e2.1}, and \eqref{e3.8}, we deduce that
\begin{align}\label{e3.9}
\left\|\widetilde{\mathcal{N}}(\nabla u\mathbf{1}_E)\right\|_{L^{p}(\partial \Omega,\sigma)}
&\lesssim \left\|\widetilde{\mathcal{C}}_1(\delta h)\right\|_{L^{p}(\partial\Omega,\sigma)}
\left\|\widetilde{\mathcal{N}}(v)\right\|_{L^{p'}(\partial\Omega,\sigma)}\notag\\
&\quad+\left\|\widetilde{\mathcal{C}}_1(\boldsymbol{F})\right\|_{L^{p}(\partial\Omega,\sigma)}
\left\|\widetilde{\mathcal{N}}(\delta\nabla v)\right\|_{L^{p'}(\partial\Omega,\sigma)}.
\end{align}
Moreover, applying Caccioppoli's inequality in Lemma \ref{Cacciopoli}, we find that, for any $x\in\partial\Omega$,
\begin{equation}\label{e3.10}
\widetilde{\mathcal{N}}(\delta\nabla v)(x)\lesssim\widetilde{\mathcal{N}}^*(v)(x)+\widetilde{\mathcal{C}}_1(\delta\boldsymbol{G})(x),
\end{equation}
where $\widetilde{\mathcal{N}}^*$ is a non-tangential maximal function with cones of larger aperture than $\widetilde{\mathcal{N}}$.
Therefore, by \eqref{e3.9}, \eqref{e3.10}, Lemma \ref{TentEqui},
the solvability of $({\rm PR}_{p'})_{L^*}$, and \eqref{e3.7}, we conclude that
\begin{align*}
\left\|\widetilde{\mathcal{N}}(\nabla u\mathbf{1}_E)\right\|_{L^{p}(\partial \Omega,\sigma)}
&\lesssim \left[\left\|\widetilde{\mathcal{C}}_1(\delta h)\right\|_{L^{p}(\partial\Omega,\sigma)}
+\left\|\widetilde{\mathcal{C}}_1(\boldsymbol{F})\right\|_{L^{p}(\partial\Omega,\sigma)}\right]\\ \notag
&\quad\times\left[\left\|\widetilde{\mathcal{N}}(v)\right\|_{L^{p'}(\partial\Omega,\sigma)}
+\left\|\widetilde{\mathcal{C}}_1(\delta\boldsymbol{G})\right\|_{L^{p'}(\partial\Omega,\sigma)}\right]\\ \nonumber
&\lesssim \left[\left\|\widetilde{\mathcal{C}}_1(\delta h)\right\|_{L^{p}(\partial\Omega,\sigma)}
+\left\|\widetilde{\mathcal{C}}_1(\boldsymbol{F})\right\|_{L^{p}(\partial\Omega,\sigma)}\right]
\left\|\widetilde{\mathcal{C}}_1(\delta\boldsymbol{G})\right\|_{L^{p'}(\partial\Omega,\sigma)}\\ \nonumber
&\lesssim \left\|\widetilde{\mathcal{C}}_1(\delta h)\right\|_{L^{p}(\partial\Omega,\sigma)}
+\left\|\widetilde{\mathcal{C}}_1(\boldsymbol{F})\right\|_{L^{p}(\partial\Omega,\sigma)}.
\end{align*}
Letting $E\uparrow\Omega$ in the above inequality, we then find that
\eqref{e3.6} holds. This finishes the proof of $(a)\Rightarrow(b)$
and hence (i).

Now, we show (ii). We only give the proof of $(e)\Rightarrow(d)$ because the proof of $(d)\Rightarrow(e)$ is similar.
Let $\boldsymbol{F}\in L_{\rm c}^{\infty}(\Omega;\mathbb{R}^n)$ and $u$ be the weak solution to the Robin problem \eqref{e3.1}.
From \eqref{e2.2} and Lemma \ref{Dense}, we infer that there exists $\boldsymbol{G}\in L_{\rm c}^{\infty}(\Omega;\mathbb{R}^n)$ such that
\begin{equation}\label{e3.11}
\left\|\widetilde{\mathcal{C}}_1(\delta\boldsymbol{G})\right\|_{L^{p}(\partial\Omega,\sigma)}\lesssim1
\end{equation}
and
\begin{equation}\label{e3.12}
\left\|\widetilde{\mathcal{N}}(\delta\nabla u)\right\|_{L^{p'}(\partial \Omega,\sigma)}
\lesssim \int_{\Omega}\delta\nabla u\cdot\boldsymbol{G}\,dx.
\end{equation}
Let $v$ be a weak solution to the Robin problem
\begin{equation*}
\begin{cases}
Lv=-\mathrm{div\,}(\delta\boldsymbol{G})\ \ & \text{in}\ \ \Omega,\\
\displaystyle\frac{\partial v}{\partial\boldsymbol{\nu}}+\alpha  v=\delta\boldsymbol{G}\cdot\boldsymbol{\nu} \ \ & \text{on}\ \ \partial\Omega.
\end{cases}
\end{equation*}
Then we have
\begin{align*}
\int_{\Omega}\delta\nabla u\cdot\boldsymbol{G}\,dx
&=\int_{\Omega}\nabla u\cdot\delta\boldsymbol{G}\,dx\\ \nonumber
&=\int_{\Omega}A\nabla v\cdot\nabla u\,dx+\int_{\partial\Omega}\alpha  vu\,d\sigma\\ \nonumber
&=\int_{\Omega} A^T\nabla u\cdot\nabla v\,dx+\int_{\partial\Omega}\alpha  vu\,d\sigma
=\int_{\Omega}\boldsymbol{F}\cdot\nabla v\,dx,
\end{align*}
which, together with \eqref{e2.1}, the solvability of $({\rm wPRR}_p)_L$, \eqref{e3.11}, and \eqref{e3.12}, further yields that
\begin{align*}
\left\|\widetilde{\mathcal{N}}(\delta\nabla u)\right\|_{L^{p'}(\partial \Omega,\sigma)}
&\lesssim\left\|\widetilde{\mathcal{C}}_1(\delta\boldsymbol{F})\right\|_{L^{p'}(\partial\Omega,\sigma)}
\left\|\widetilde{\mathcal{N}}(\nabla v)\right\|_{L^{p}(\partial\Omega,\sigma)}\\ \nonumber
&\lesssim\left\|\widetilde{\mathcal{C}}_1(\delta\boldsymbol{F})\right\|_{L^{p'}(\partial\Omega,\sigma)}
\left\|\widetilde{\mathcal{C}}_1(\delta\boldsymbol{G})\right\|_{L^{p}(\partial\Omega,\sigma)}
\lesssim\left\|\widetilde{\mathcal{C}}_1(\delta\boldsymbol{F})\right\|_{L^{p'}(\partial\Omega,\sigma)}.
\end{align*}
Therefore, the problem $({\rm wPR}_{p'})_{L^*}$ is solvable.
This finishes the proof of (ii) and hence Theorem \ref{PRandPRR}.
\end{proof}

Next, we study the relationship among the Poisson--Robin problem, the localization property, and the weak Poisson--Robin problem,
that is, to prove Theorem \ref{StrongLocWeak}.
For this purpose, we first establish the following three successive lemmas.

\begin{lemma}\label{strongerLoc}
Let $n\ge2$, $\Omega\subset\mathbb{R}^n$ be a bounded one-sided {\rm CAD}, $p\in(1,\infty)$, $p'$ be the H\"older conjugate of $p$,
$L:=-\mathrm{div}(A\nabla\cdot)$ be a uniformly elliptic operator, and $\alpha$ satisfy \eqref{e1.3} and $\alpha\in L^p(\partial\Omega,\sigma)$ with $\|\alpha \|_{L^p(B(z,r)\cap\partial\Omega,\sigma)}\le Cr^{(n-1)/p}$ for any $z\in\partial\Omega$ and $r\in(0,\mathrm{diam\,}(\Omega))$.
If $({\rm PR}_{p'})_{L^*}$ is solvable, then
there exists a positive constant $C$ such that, for any ball $B:=B(x,r)$ with $x\in\partial\Omega$ and $r\in(0,\mathrm{diam\,}(\Omega))$
and for any local solution $u\in W^{1,2}(2B\cap\Omega)$
to $Lu=0$ in $2B\cap\Omega$ with the Robin boundary condition $\frac{\partial u}{\partial\boldsymbol{\nu}}+\alpha  u=0$ on $2B\cap\partial\Omega$,
\begin{equation}\label{e3.13}
\left\|\widetilde{\mathcal{N}}(|\nabla u|\mathbf{1}_B)\right\|_{L^{p}(\partial \Omega,\sigma)}
\le C r^{\frac{n-1}{p}}\left(\fint_{2B\cap\Omega}|\nabla u|\,dx+\fint_{2B\cap\Omega}|u|\,dx\right).
\end{equation}
\end{lemma}

\begin{proof}
Let $B:=B(x,r)$ with $x\in\partial\Omega$ and $r\in(0,\mathrm{diam\,}(\Omega))$, and let $u\in W^{1,2}(2B\cap\Omega)$ be a weak solution to
$Lu=0$ in $2B\cap\Omega$ with the Robin boundary condition $\frac{\partial u}{\partial\boldsymbol{\nu}}+\alpha  u=0$ on $2B\cap\partial\Omega$.
Take $\phi\in C_{\rm c}^{\infty}(\mathbb{R}^n)$ satisfies that $\phi\equiv1$ in $\frac{6}{5}B$, $\mathrm{supp\,}(\phi)\subset\frac{5}{4}B$, and $|\nabla\phi|\lesssim\frac{1}{r}$.
By \eqref{e2.2} and Lemma \ref{Dense}, we conclude that there exists $\boldsymbol{F}\in L_{\rm c}^{\infty}(\Omega;\mathbb{R}^n)$ such that
\begin{equation}\label{e3.14}
\left\|\widetilde{\mathcal{C}}_1(\delta\boldsymbol{F})\right\|_{L^{p'}(\partial\Omega,\sigma)}\lesssim1
\end{equation}
and
\begin{equation}\label{e3.15}
\left\|\widetilde{\mathcal{N}}(\nabla u\mathbf{1}_B)\right\|_{L^{p}(\partial \Omega,\sigma)}\lesssim \int_{\Omega}\mathbf{1}_B\nabla u\cdot\boldsymbol{F}\,dx.
\end{equation}
Let $\boldsymbol{F}_B:=\boldsymbol{F}\mathbf{1}_B$ and $c_u:=\fint_{\frac{5}{4}B\cap\Omega} u\,dx$.
From the assumptions that $\mathrm{supp\,}(F_B)\subset B$ and $\nabla\phi\equiv0$ on $\frac{6}{5}B$,
we deduce that
$$
\int_{\Omega} (u-c_u)\boldsymbol{F}_B\cdot\nabla\phi\,dx=0,
$$
which, combined with \eqref{e3.15}, further implies that
\begin{align}\label{e3.16}
\left\|\widetilde{\mathcal{N}}(\nabla u\mathbf{1}_B)\right\|_{L^{p}(\partial \Omega,\sigma)}
&\lesssim \int_{\Omega}\nabla (u-c_u)\cdot\boldsymbol{F}_B\phi\,dx\nonumber\\
&=\int_{\Omega}\nabla((u-c_u)\phi)\cdot\boldsymbol{F}_B\,dx
=:\mathrm{I}_1.
\end{align}
Let $v$ be the weak solution to the Robin problem \eqref{e3.1} with $\boldsymbol{F}$ replaced by $\boldsymbol{F}_B$. We then have
\begin{align}\label{e3.17}
\mathrm{I}_1&=\int_{\Omega}A\nabla((u-c_u)\phi)\cdot\nabla v\,dx+\int_{\partial\Omega}\alpha  v (u-c_u)\phi\,d\sigma\nonumber\\ \nonumber
&=\int_{\Omega}\phi A\nabla u\cdot\nabla v\,dx+\int_{\Omega}(u-c_u) A\nabla\phi\cdot\nabla v\,dx+\int_{\partial\Omega}\alpha  v (u-c_u)\phi\,d\sigma\\ \nonumber
&=\int_{\Omega}A\nabla u\cdot\nabla(\phi v)\,dx-\int_{\Omega}
v A\nabla u\cdot\nabla\phi\,dx\\
&\quad+\int_{\Omega}(u-c_u) A\nabla\phi\cdot\nabla v\,dx+\int_{\partial\Omega}\alpha  v (u-c_u)\phi\,d\sigma.
\end{align}
By \eqref{e3.17} and the assumption that $Lu=0$ in $2B\cap\Omega$ with $\frac{\partial u}{\partial\boldsymbol{\nu}}+\alpha  u=0$ on $2B\cap\partial\Omega$,
we find that
\begin{align}\label{e3.18}
\mathrm{I}_1&=-\int_{\Omega}
v A\nabla u\cdot\nabla\phi\,dx+\int_{\Omega}(u-c_u)A\nabla\phi\cdot\nabla v\,dx-\int_{\partial\Omega}\alpha  v c_u\phi\,d\sigma\notag\\
&=:\mathrm{I}_2+\mathrm{I}_3+\mathrm{I}_4.
\end{align}

Now, we estimate the terms $\mathrm{I}_2$, $\mathrm{I}_3$, and $\mathrm{I}_4$, respectively. For $\mathrm{I}_2$, from the fact that $v$ is a solution to $L^*v=0$ in $(2B\setminus B)\cap\Omega$ with $\frac{\partial v}
{\partial\boldsymbol{\nu}}+\alpha  v=0$ on $(2B\setminus B)\cap\partial\Omega$, H\"older's inequality, Lemmas \ref{Cacciopoli} and \ref{Moser},
the definition of $\widetilde{\mathcal{N}}$, the solvability of $({\rm PR}_{p'})_{L^*}$, and \eqref{e3.14}, we deduce that
\begin{align}\label{e3.19}
\left|\mathrm{I}_2\right|
&\lesssim r^{n-1}\fint_{\frac{5}{4}B\cap\Omega}|\nabla u|\,dx \sup_{x\in(\frac{5}{4}B\setminus\frac{6}{5}B)\cap\Omega}|v(x)|\nonumber\\
&\lesssim r^{n-1}\fint_{2B\cap\Omega}|\nabla u|\,dx \fint_{(2B\setminus B)\cap\Omega}| v|\,dx\nonumber\\
&\lesssim r^{n-1} r^{\frac{1-n}{p'}}\fint_{2B\cap\Omega}|\nabla u|\,dx\left\|\widetilde{\mathcal{N}}(v)\right\|_{L^{p'}(\partial\Omega,\sigma)}\nonumber\\ &\lesssim r^{\frac{n-1}{p}}\fint_{2B\cap\Omega}|\nabla u|\,dx\left\|\widetilde{\mathcal{C}}_1(\delta\boldsymbol{F}_B)\right\|_{L^{p'}(\partial\Omega,\sigma)}
\lesssim r^{\frac{n-1}{p}}\fint_{2B\cap\Omega}|\nabla u|\,dx.
\end{align}
For $\mathrm{I}_3$, by H\"older's inequality, Lemmas \ref{Cacciopoli}, \ref{Moser}, and \ref{l2.2}, Corollary \ref{c2.1}, the definition of $\widetilde{\mathcal{N}}$,
the solvability of $({\rm PR}_{p'})_{L^*}$, and \eqref{e3.14}, we conclude that
\begin{align}\label{e3.20}
\left|\mathrm{I}_3\right|
&\lesssim r^{n-1}\left[\fint_{(\frac{5}{4}B\setminus\frac{6}{5}B)\cap\Omega}|\nabla v|^2\,dx\right]^{\frac{1}{2}}
\left[\fint_{\frac{5}{4}B\cap\Omega}|u-c_u|^2\,dx\right]^{\frac{1}{2}}\nonumber\\ \nonumber
&\lesssim r^{n-1}\left[\fint_{(\frac{4}{3}B\setminus\frac{7}{6}B)\cap\Omega}|v|^2\,dx\right]^{\frac{1}{2}}
\left[\fint_{\frac{4}{3}B\cap\Omega}|\nabla u|^2\,dx\right]^{\frac{1}{2}}\\ \nonumber
&\lesssim r^{n-1}\sup_{x\in(\frac{4}{3}B\setminus\frac{7}{6}B)\cap\Omega}|v(x)|
\fint_{2B\cap\Omega}|\nabla u|\,dx\\ \nonumber
&\lesssim r^{n-1}\fint_{(2B\setminus B)\cap\Omega}|v|\,dx
\fint_{2B\cap\Omega}|\nabla u|\,dx\\ \nonumber
&\lesssim r^{\frac{n-1}{p}}\left\|\widetilde{\mathcal{N}}(v)\right\|_{L^{p'}(\partial\Omega,\sigma)}
\fint_{2B\cap\Omega}|\nabla u|\,dx\\
&\lesssim r^{\frac{n-1}{p}}\fint_{2B\cap\Omega}|\nabla u|\,dx\left\|\widetilde{\mathcal{C}}_1(\delta\boldsymbol{F}_B)\right\|_{L^{p'}(\partial\Omega,\sigma)}
\lesssim r^{\frac{n-1}{p}}\fint_{2B\cap\Omega}|\nabla u|\,dx.
\end{align}
For $\mathrm{I}_4$, from H\"older's inequality,
the solvability of $({\rm PR}_{p'})_{L^*}$, \eqref{e3.14},
and the assumption that $\|\alpha \|_{L^p(B(z,r)\cap\partial\Omega,\sigma)}\lesssim r^{(n-1)/p}$ for any $z\in\partial\Omega$
and $r\in(0,\mathrm{diam\,}(\Omega))$, it follows that
\begin{align*}
|\mathrm{I}_4|
&\lesssim |c_u|\int_{(\frac{5}{4}B\setminus\frac{6}{5}B)\cap\partial\Omega}\alpha \widetilde{\mathcal{N}}(v)\,d\sigma\\ \nonumber
&\lesssim \fint_{2B\cap\Omega}|u|\,dx\|\alpha \|_{L^{p}(2B\cap\partial\Omega,\sigma)}
\left\|\widetilde{\mathcal{N}}(v)\right\|_{L^{p'}(\partial\Omega,\sigma)}\\ \nonumber
&\lesssim r^{\frac{n-1}{p}}\fint_{2B\cap\Omega}|u|\,dx\left\|\widetilde{\mathcal{C}}_1(\delta\boldsymbol{F}_B)
\right\|_{L^{p'}(\partial\Omega,\sigma)}
\lesssim r^{\frac{n-1}{p}}\fint_{2B\cap\Omega}|u|\,dx.
\end{align*}
This, combined with \eqref{e3.15}, \eqref{e3.16}, \eqref{e3.18}, \eqref{e3.19}, and \eqref{e3.20}, further implies that \eqref{e3.13} holds.
This finishes the proof of Lemma \ref{strongerLoc}.
\end{proof}

Next, we show that the local property $({\rm LocR}_p)_L$ has the
self-improvement property
and implies the solvability of the weak $L^p$ Poisson--Robin problem.

\begin{lemma}\label{selfimprove}
Let $n\ge2$, $\Omega\subset\mathbb{R}^n$ be a bounded one-sided ${\rm CAD}$, $L:=-\mathrm{div\,}(A\nabla\cdot)$ be a uniformly elliptic operator,
and $p\in(1,\infty)$. If $({\rm LocR}_p)_L$ holds, then there exists a positive constant $\varepsilon$, depending only on $p$, $n$, $L$, and $\Omega$,
such that $({\rm LocR}_{q})_L$ holds for any $q\in[1,p+\varepsilon)$.
\end{lemma}

\begin{proof}
Let $B:=B(x,r)$ with $x\in\partial\Omega$ and $r\in(0,\mathrm{diam\,}(\Omega))$ and let $u\in W^{1,2}(2B\cap\Omega)$ be a weak solution to
$Lu=0$ in $2B\cap\Omega$ with the Robin boundary condition $\frac{\partial u}{\partial\boldsymbol{\nu}}+\alpha  u=0$ on $2B\cap\partial\Omega$.
Assume that $({\rm LocR}_p)_L$ holds. We first prove $({\rm LocR}_{q})_L$ holds for any $q\in[1,p]$.

By the fact that $\widetilde{\mathcal{N}}(|\nabla u|\mathbf{1}_B)$ is supported in $5B\cap\partial\Omega$, H\"older's inequality, and
the assumption that $({\rm Loc}_{p})_L$ holds, we find that
\begin{align*}
r^{-\frac{n-1}{q}}\left\|\widetilde{\mathcal{N}}\left(|\nabla u|\mathbf{1}_B\right)\right\|_{L^{q}(\partial\Omega,\sigma)}
&\lesssim \left\{\fint_{5B\cap\partial\Omega}\left[\widetilde{\mathcal{N}}\left(|\nabla u|\mathbf{1}_B\right)\right]^{q}\,d\sigma\right\}^{\frac{1}{q}}\\ \notag
&\lesssim \left\{\fint_{5B\cap\partial\Omega}\left[\widetilde{\mathcal{N}}(|\nabla u|\mathbf{1}_B)\right]^{p}\,d\sigma\right\}^{\frac{1}{p}}\\ \nonumber
&\lesssim \fint_{2B\cap\Omega}|\nabla u|\,dx+\fint_{2B\cap\Omega}|u|\,dx.
\end{align*}
Thus, $({\rm LocR}_{q})_L$ holds for any $q\in[1,p]$.

Next, we prove that there exists some $\varepsilon\in(0,\infty)$, depending only on $p$, $n$, $L$, and $\Omega$, such that
\begin{equation}\label{e3.21}
\left\|\widetilde{\mathcal{N}}(|\nabla u|\mathbf{1}_B)\right\|_{L^{p+\varepsilon}(\partial\Omega,\sigma)}\lesssim r^{\frac{n-1}{p+\varepsilon}}
\left(\fint_{2B\cap\Omega}|\nabla u|\,dx+\fint_{2B\cap\Omega}|u|\,dx\right),
\end{equation}
which further implies that $({\rm LocR}_q)_L$ holds for any $q\in(p,p+\varepsilon)$.
We first claim that, for any $p\in(1,\infty)$,
\begin{equation}\label{e3.22}
\left\|\widetilde{\mathcal{N}}(|u|\mathbf{1}_B)\right\|_{L^{p}(\partial \Omega,\sigma)}\lesssim r^{\frac{n-1}{p}}\fint_{2B\cap\Omega}|u|\,dx.
\end{equation}
Indeed, from \eqref{e2.2}, Lemma \ref{Dense}, the definition of $\widetilde{\mathcal{C}}_1$, and Lemma \ref{Moser}, it follows that
there exists $g\in L_{\rm c}^{\infty}(\Omega)$ such that
$$
\left\|\widetilde{\mathcal{C}}_1(\delta g)\right\|_{L^{p'}(\partial\Omega,\sigma)}\lesssim1
$$
and
\begin{align*}
\left\|\widetilde{\mathcal{N}}(u\mathbf{1}_B)\right\|_{L^{p}(\partial \Omega,\sigma)}
&\lesssim \int_{\Omega}\mathbf{1}_Bug\,dx\lesssim r^{n}\sup_{x\in B\cap\Omega}|u(x)|\fint_{B\cap\Omega}|g|\,dx\\ \nonumber
&\lesssim r^{n-1}\fint_{2B\cap\Omega}|u|\,dx\inf_{x\in B\cap\partial\Omega}\widetilde{\mathcal{C}}_1(\delta g)(x)\\ \nonumber
&\lesssim r^{n-1}r^{\frac{1-n}{p'}}\fint_{2B\cap\Omega}|u|\,dx\left\|\widetilde{\mathcal{C}}_1(\delta g)\right\|_{L^{p'}(\partial\Omega,\sigma)}
\lesssim r^{\frac{n-1}{p}}\fint_{2B\cap\Omega}|u|\,dx.
\end{align*}
Furthermore, for any $y\in\partial\Omega$, define
$$
g(y):=\widetilde{\mathcal{N}}(|\nabla u|\mathbf{1}_B)(y)+\widetilde{\mathcal{N}}(|u|\mathbf{1}_B)(y).
$$
We claim that  there exists a positive constant $C$, independent of $B$ and $u$, such that, for any ball $\widetilde{B}$
centered at the boundary $\partial\Omega$ that satisfies $24\widetilde{B}\subset10B$,
\begin{equation}\label{e3.23}
\left(\fint_{\widetilde{B}\cap\partial\Omega}g^p\,d\sigma\right)^{\frac{1}{p}}\le C\fint_{12\widetilde{B}\cap\partial\Omega}g\,d\sigma.
\end{equation}
If \eqref{e3.23} holds, by \cite[Proposition 1.2, p.\,122]{gia83}, we conclude that there exists $\varepsilon\in(0,\infty)$ such that
\begin{equation}\label{e3.24}
\left(\fint_{5B\cap\partial\Omega}g^{p+\varepsilon}\,d\sigma
\right)^{\frac{1}{p+\varepsilon}}\lesssim\left(\fint_{10B\cap\partial\Omega}g^p
\,d\sigma\right)^{\frac{1}{p}}.
\end{equation}
Then, from \eqref{e3.24}, the assumption that $({\rm Loc}_{p})_{L}$ holds, and \eqref{e3.22}, it follows that
\begin{align*}
&r^{-\frac{n-1}{p+\varepsilon}}\left\|\widetilde{\mathcal{N}}(|\nabla u|\mathbf{1}_B)\right\|_{L^{p+\varepsilon}(\partial\Omega,\sigma)}\\
&\quad\lesssim \left\{\fint_{5B\cap\partial\Omega}\left[\widetilde{\mathcal{N}}
(|\nabla u|\mathbf{1}_B)\right]^{p+\varepsilon}\,d\sigma\right\}^{\frac{1}{p+\varepsilon}}\\ \nonumber
&\quad\lesssim \left(\fint_{5B\cap\partial\Omega}g^{p+\varepsilon}\,d\sigma\right)^{\frac{1}{p+\varepsilon}}
\lesssim\left(\fint_{10B\cap\partial\Omega}g^p\,d\sigma\right)^{\frac{1}{p}}\\ \nonumber
&\quad\lesssim \left\{\fint_{10B\cap\partial\Omega}\left[\widetilde{\mathcal{N}}(|\nabla u|\mathbf{1}_B)\right]^{p}\,d\sigma\right\}^{\frac{1}{p}}
+\left\{\fint_{10B\cap\partial\Omega}\left[\widetilde{\mathcal{N}}
(|u|\mathbf{1}_B)\right]^{p}\,d\sigma\right\}^{\frac{1}{p}}\\ \nonumber
&\quad\lesssim \fint_{2B\cap\Omega}|\nabla u|\,dx+\fint_{2B\cap\Omega}|u|\,dx,
\end{align*}
which implies that \eqref{e3.21} holds.

Now, we return to prove \eqref{e3.23}. Let $z\in\partial\Omega$ and $\widetilde{B}:=B(z,s)$ satisfy $24\widetilde{B}\subset10B$.
For any $y\in \widetilde{B}\cap\partial\Omega$, we decompose $g(y)$ into
\begin{align*}
g(y)
&\le\widetilde{\mathcal{N}}\left(|\nabla u|\mathbf{1}_{\frac{6}{5}\widetilde{B}}\right)(y)+\widetilde{\mathcal{N}}\left(|\nabla u|\mathbf{1}_{B\setminus\frac{6}{5}\widetilde{B}}\right)(y)
+\widetilde{\mathcal{N}}\left(|u|\mathbf{1}_{\frac{6}{5}\widetilde{B}}\right)(y)
+\widetilde{\mathcal{N}}\left(|u|\mathbf{1}_{B\setminus\frac{6}{5}\widetilde{B}}\right)(y)\\ \nonumber
&=:g_1(y)+g_2(y)+g_3(y)+g_4(y).
\end{align*}
For $g_1$, by the assumption that $({\rm LocR}_p)_L$ holds, we find that
\begin{align}\label{e3.25}
\left(\fint_{\widetilde{B}\cap\partial\Omega}g_1^p\,d\sigma\right)^{\frac{1}{p}}
&\lesssim s^{-\frac{n-1}{p}}\|g_1\|_{L^p(\partial\Omega,\sigma)}
\lesssim s^{-\frac{n-1}{p}}\left\|\widetilde{\mathcal{N}}\left(|\nabla u|
\mathbf{1}_{\frac{6}{5}\widetilde{B}}\right)\right\|_{L^p(\partial\Omega,\sigma)}\notag\\
&\lesssim \fint_{\frac{12}{5}\widetilde{B}\cap\Omega}|\nabla u|\,dx+\fint_{\frac{12}{5}\widetilde{B}\cap\Omega}|u|\,dx.
\end{align}
A direct calculation yields that, for any $y\in\partial\Omega$,
$$\int_{\gamma(y)}\mathbf{1}_{\frac{12}{5}\widetilde{B}}(x)[\delta(x)]^{1-n}\,dx\lesssim s,$$
which further implies that
\begin{align}\label{e3.26}
\fint_{\frac{12}{5}\widetilde{B}\cap\Omega}|\nabla u|\,dx
&\lesssim s^{-n}\int_{\partial\Omega}\int_{\gamma(y)}\left(\fint_{B(x,\frac{1}{4}\delta(x))}|\nabla u(w)|^2
\mathbf{1}_{\frac{12}{5}\widetilde{B}}(w)\,dw\right)^{\frac{1}{2}}\nonumber\\ \notag
&\quad\times\mathbf{1}_{\frac{12}{5}\widetilde{B}}(x)[\delta(x)]^{1-n}\,dxd\sigma(y)\\
&\lesssim s^{1-n}\int_{\partial\Omega}\widetilde{\mathcal{N}}\left(|\nabla u|\mathbf{1}_{\frac{12}{5}\widetilde{B}}\right)\,d\sigma
\lesssim \fint_{12\widetilde{B}\cap\partial\Omega}\widetilde{\mathcal{N}}\left(|\nabla u|\mathbf{1}_{\frac{12}{5}\widetilde{B}}\right)\,d\sigma.
\end{align}
Similarly, we also have
$$
\fint_{\frac{12}{5}\widetilde{B}\cap\Omega}|u|\,dx
\lesssim \fint_{12\widetilde{B}\cap\partial\Omega}\widetilde{\mathcal{N}}
\left(|u|\mathbf{1}_{\frac{12}{5}\widetilde{B}}\right)\,d\sigma.
$$
From this, \eqref{e3.25}, \eqref{e3.26}, and $24\widetilde{B}\subset10B$, we deduce that
\begin{align}\label{e3.27}
\left(\fint_{\widetilde{B}\cap\partial\Omega}g_1^p\,d\sigma\right)^{\frac{1}{p}}
&\lesssim \fint_{12\widetilde{B}\cap\partial\Omega}\widetilde{\mathcal{N}}\left(|\nabla u|\mathbf{1}_{\frac{12}{5}\widetilde{B}}\right)\,d\sigma
+\fint_{12\widetilde{B}\cap\partial\Omega}\widetilde{\mathcal{N}}
\left(|u|\mathbf{1}_{\frac{12}{5}\widetilde{B}}\right)\,d\sigma\nonumber\\
&\lesssim \fint_{12\widetilde{B}\cap\partial\Omega}\widetilde{\mathcal{N}}\left(|\nabla u|\mathbf{1}_{B}\right)\,d\sigma
+\fint_{12\widetilde{B}\cap\partial\Omega}\widetilde{\mathcal{N}}
\left(|u|\mathbf{1}_{B}\right)\,d\sigma
\lesssim \fint_{12\widetilde{B}\cap\partial\Omega}g\,d\sigma.
\end{align}
For $g_2$, similar to an argument used in \cite[p.\,31]{fl24}, we conclude that, for any $y\in\widetilde{B}\cap\partial\Omega$,
\begin{align}\label{e3.28}
g_{2}(y)\lesssim\fint_{2\widetilde{B}\cap\partial\Omega}\widetilde{\mathcal{N}}^{a}\left(|\nabla u|\mathbf{1}_{B}\right)\,d\sigma,
\end{align}
where $\widetilde{\mathcal{N}}^{a}$ denotes the modified non-tangential maximal function associated with some aperture $a$ with $a\in(2,\infty)$
being a constant independent of $u$, $\widetilde{B}$, and $y$.
Thus, by Lemma \ref{TentEqui} and \eqref{e3.28}, we find that, for any $y\in\widetilde{B}\cap\partial\Omega$,
$$
g_{2}(y)\lesssim\fint_{2\widetilde{B}\cap\partial\Omega}\widetilde{\mathcal{N}}\left(|\nabla u|\mathbf{1}_{B}\right)\,d\sigma
\lesssim\fint_{12\widetilde{B}\cap\partial\Omega}g\,d\sigma,
$$
which further implies that
\begin{equation}\label{e3.29}
\left(\fint_{\widetilde{B}\cap\partial\Omega}g_2^p\,d\sigma\right)^{\frac{1}{p}}
\lesssim\fint_{12\widetilde{B}\cap\partial\Omega}g\,d\sigma.
\end{equation}
Moreover, applying \eqref{e3.22} and an argument similar to that used in \eqref{e3.27} and \eqref{e3.29}, we conclude that, for $i\in\{3,4\}$,
$$
\left(\fint_{\widetilde{B}\cap\partial\Omega}g_i^p\,d\sigma\right)^{\frac1p}
\lesssim\fint_{12\widetilde{B}\cap\partial\Omega}g\,d\sigma,
$$
which, combined with \eqref{e3.27} and \eqref{e3.29}, further implies that \eqref{e3.23} holds.
This finishes the proof of Lemma \ref{selfimprove}.
\end{proof}

\begin{lemma}\label{LocwPR}
Let $n\ge3$, $\Omega\subset\mathbb{R}^n$ be a bounded one-sided {\rm CAD}, $p\in(1,\infty)$, $p'$ be the H\"older conjugate of $p$,
$L:=-\mathrm{div}(A\nabla\cdot)$ be a uniformly elliptic operator, and $\alpha$ be the same as in \eqref{e1.3}.
If $({\rm LocR}_p)_L$ holds, then $({\rm wPR}_{p'})_{L^*}$ is solvable.
\end{lemma}

\begin{proof}
We prove the present lemma by applying the representation formula \eqref{e2.11} and borrowing some ideas from the proofs of
\cite[Theorem 1.11]{mpt22}, \cite[Lemma 5.9]{fl24}, and \cite[Lemma 2.10]{kp96}.

Let $\boldsymbol{F}\in L^{\infty}_{\rm c}(\Omega;\mathbb{R}^n)$ and $u$ be the weak solution to the Robin problem
\begin{equation}\label{e3.30}
\begin{cases}
L^*u=-\mathrm{div\,}\boldsymbol{F}\ \ & \text{in}\ \ \Omega,\\
\displaystyle\frac{\partial u}{\partial\boldsymbol{\nu}}+\alpha  u=\boldsymbol{F}\cdot\boldsymbol{\nu} \ \ & \text{on}\ \ \partial\Omega.
\end{cases}
\end{equation}
Then, from \eqref{e2.11} and Remark \ref{r2.1}, we infer that, for any $x\in\Omega$,
\begin{equation*}
u(x)=\int_{\Omega}\nabla_{y}G_R(y,x)\cdot\boldsymbol{F}(y)\,dy.
\end{equation*}
By the assumption that $({\rm LocR}_p)_L$ holds and Lemma \ref{selfimprove}, we find that there exists $p_1\in(p,\infty)$
such that $({\rm Loc}_{p_1})_L$ holds.
Let $a\in(1,\frac{3}{2})$. We first claim that, for any $z\in\partial\Omega$,
\begin{equation}\label{e3.31}
\widetilde{\mathcal{N}}^{a,\frac{1}{16}}(\delta\nabla u)(z)
\lesssim\left\{\mathcal{M}\left(\left[\widetilde{\mathcal{C}}_1^{\frac{5}{8}}( \delta\boldsymbol{F})\right]^{p_1'}\right)(z)\right\}^{\frac{1}{p_1'}}
+\widetilde{\mathcal{A}}_1^{a^*,\frac{1}{16}}(\delta\boldsymbol{F})(z),
\end{equation}
where $p_1'\in(1,\infty)$ is the H\"older conjugate of $p_1$ and $a^*\in(a,\infty)$ is a constant determined later.
If \eqref{e3.31} holds, then, from \eqref{e3.31}, Lemma \ref{TentEqui}, and the $L^{\frac{p'}{p_1'}}$-boundedness of the
Hardy--Littlewood maximal operator $\mathcal{M}$ (see Lemma \ref{l2.1}),
we deduce that
$$
\left\|\widetilde{\mathcal{N}}(\delta\nabla u)\right\|_{L^{p'}(\partial\Omega,\sigma)}\lesssim
\left\|\widetilde{\mathcal{C}}_1(\delta\boldsymbol{F})\right\|_{L^{p'}(\partial\Omega,\sigma)},
$$
which implies that the problem $({\rm wPR}_{p'})_{L^*}$ is solvable.

Now, we turn to prove \eqref{e3.31}. Fix $z\in\partial\Omega$ and $y\in\gamma(z)$.
Then, for any $x\in\Omega$, we decompose $u(x)$ as
\begin{align}\label{e3.32}
u(x)=&\,u_1(x)+u_2(x)+u_3(x)\nonumber\\
:=&\,\int_{2B_y}\boldsymbol{F}(w)\cdot\nabla_{w}G_R(w,x)\,dw+
\int_{\gamma_{a*,y}(z)\setminus2B_y}\cdots+\int_{\Omega\setminus\gamma_{a*,y}(z)}
\cdots,
\end{align}
where $B_y:=B(y,\frac{1}{4}\delta(y))$ and
$$
\gamma_{a*,y}(z):=\left\{w\in\Omega:\frac{\delta(y)}{a^*}<|w-z|<a^*\delta(w)\right\}
$$
is a cone, with aperture $a^*$ to be determined later, removing a small ball centered at $z$.

We first deal with $u_1$. Notice that $u_1$ is a weak solution to the Robin problem \eqref{e3.30}
with $\boldsymbol{F}$ replaced by $\boldsymbol{F}\mathbf{1}_{2B_y}$.
Then, by Remark \ref{r1.1} and the fact that, for any $x\in\frac{1}{8}B_y$, $2B_y\subset\frac{5}{2}B_x$, we conclude that
\begin{align}\label{e3.33}
&\left(\fint_{\frac{1}{4}B_y}|\delta\nabla u_1|^2\,dw\right)^{\frac{1}{2}}\nonumber\\ \notag
&\quad\lesssim[\delta(y)]^{1-\frac{n}{2}}\left(\int_{\Omega}|\nabla u_1|^2\,dw\right)^{\frac12}\\ \notag
&\quad\lesssim[\delta(y)]^{1-\frac{n}{2}}\left(\int_{\Omega}|\boldsymbol{F}
\mathbf{1}_{2B_y}|^2\,dw\right)^{\frac{1}{2}}
\lesssim\left(\fint_{2B_y}|\delta\boldsymbol{F}|^2\,dw\right)^{\frac{1}{2}}\\ \nonumber
&\quad\lesssim\fint_{\frac{1}{8}B_y}\left(\fint_{2B_y}|\delta\boldsymbol{F}|^2\,dw
\right)^{\frac{1}{2}}\,d\xi
\lesssim\fint_{\frac{1}{8}B_y}\left(\fint_{\frac{5}{2}B_{\xi}}|\delta\boldsymbol{F}|^2
\,dw\right)^{\frac{1}{2}}\,d\xi\\
&\quad\lesssim[\delta(y)]^{1-n}\int_{B(z,2\delta(y))\cap\Omega}\left(\fint_{\frac{5}{2}B_{\xi}}
|\delta\boldsymbol{F}|^2\,dw\right)^{\frac{1}{2}}\,\frac{d\xi}{\delta(\xi)}
\lesssim\widetilde{\mathcal{C}}_1^{\frac{5}{8}}(\delta\boldsymbol{F})(z).
\end{align}

For $u_2$, since $L^*u_2=0$ in $2B_y$, it follows from Caccioppoli's inequality as in Lemma \ref{Cacciopoli} and the Moser type
estimate as in Lemma \ref{Moser} that
$$
\left(\fint_{\frac{1}{4}B_y}|\delta\nabla u_2|^2\,dw\right)^{\frac{1}{2}}
\lesssim\left(\fint_{\frac{1}{2}B_y}|u_2|^2\,dw\right)^{\frac{1}{2}}
\lesssim\sup_{w\in\frac{1}{2}B_y}|u_2(w)|,
$$
which, combined with the expression of $u_2$ as in \eqref{e3.32}, Lemma \ref{Green}(v), and the fact that,
for any $x\in\gamma_{a^*,y}(z)\setminus2B_y$ and $w\in\frac{1}{2}B_y$, $B_w\subset2a(a^*)^2B_x\setminus(\frac{1}{2}B_x)$
(see \cite[p.\,34]{fl24} for its proof), further yields that
\begin{align}\label{e3.34}
&\left(\fint_{\frac{1}{4}B_y}|\delta\nabla u_2|^2\,dw\right)^{\frac{1}{2}}\nonumber\\ \nonumber
&\quad\lesssim\sup_{w\in\frac{1}{2}B_y}|u_2(w)|
\lesssim\sup_{w\in\frac{1}{2}B_y}\left|\int_{\gamma_{a^*,y}(z)\setminus2B_y}
\boldsymbol{F}(x)\cdot\nabla_{x}G_R(x,w)\,dx\right|\\ \nonumber
&\quad\lesssim\sup_{w\in\frac{1}{2}B_y}\int_{\gamma_{a^*,y}(z)\setminus2B_y}
\fint_{\frac{1}{4}B_x}\left|\boldsymbol{F}(\xi)\cdot\nabla_{\xi}G_R(\xi,w)\right|\,d\xi dx\\ \nonumber
&\quad\lesssim\sup_{w\in\frac{1}{2}B_y}\int_{\gamma_{a^*,y}(z)\setminus2B_y}
\left(\fint_{\frac{1}{4}B_x}|\boldsymbol{F}|^2\,d\xi\right)^{\frac{1}{2}}
\left(\fint_{\frac{1}{4}B_x}\left|\nabla_{\xi}G_R(\xi,w)\right|^2\,d\xi\right)^{\frac{1}{2}}\,dx\\
&\quad\lesssim\int_{\gamma_{a^*}(z)}\left(\fint_{\frac{1}{4}B_x}|
\delta\boldsymbol{F}|^2\right)^{\frac{1}{2}}\,\frac{dx}{[\delta(x)]^n}
\lesssim\widetilde{\mathcal{A}}_1^{a^*,\frac{1}{16}}(\delta\boldsymbol{F})(z).
\end{align}

Next, we estimate $u_3$. Let $a^*\in(300,\infty)$.  Then the union of balls $\{B(x,\frac{1}{100}
|x-y|)\}_{x\in\partial\Omega}$ covers $\Omega\setminus\gamma_{a*,y}(z)$. Thus, applying Vitali's covering lemma (see, for example, \cite{s93}),
we find that there exist an index set $I$ and a non-overlapping subcollection $\{B(x_k,\frac{1}{100}|x_k-y|)\}_{k\in I}$ such that
$\{B_k:=B(x_k,\frac{1}{20}|x_k-y|)\}_{k\in I}$ is a finitely overlapping cover of $\Omega\setminus\gamma_{a*,y}(z)$.
For any $k\in I$, let
$$
D_k:=\left(B_k\cap\Omega\right)\setminus\left(\gamma_{a*,y}(z)\cup\bigcup_{i\in I,\,i<k}B_i\right).
$$
Then
$$
\Omega\setminus\gamma_{a*,y}(z)\subset\bigcup_{k\in I}D_k.
$$
Notice that, for any $k\in I$, the radius of $B_k$ is bounded from below by $\frac{1}{20}\delta(y)$, and the number of balls in $\{B_k\}_{k\in I}$
with radius between $\frac{2^m}{20}\delta(y)$ and $\frac{2^{m+1}}{20}\delta(y)$ is uniformly bounded for $m\in\mathbb{N}$.
Thus, for any given $\beta\in(0,\infty)$,
\begin{equation}\label{e3.35}
\sum_{k\in I}\left(\frac{\delta(y)}{|x_k-y|}\right)^{\beta}\lesssim 1.
\end{equation}
Moreover, it is easy to find that
\begin{equation}\label{e3.36}
\left(\fint_{\frac{1}{4}B_y}|\delta\nabla u_3|^2\,dw\right)^{\frac{1}{2}}
\lesssim\sum_{k\in I}\left(\fint_{\frac{1}{4}B_y}\left|\delta\nabla u_{3,k}\right|^2\,dw\right)^{\frac{1}{2}},
\end{equation}
where, for any $k\in I$ and any $w\in B_y/4$,
$$
u_{3,k}(w):=\int_{D_k}\boldsymbol{F}(x)\cdot\nabla_{x}G_R(x,w)\,dx.
$$
Fix $k\in I$. For $u_{3,k}$, noticing the fact that $u_{3,k}$ is a weak solution to the Robin problem \eqref{e3.30} with $\boldsymbol{F}$ replaced by $\boldsymbol{F}\mathbf{1}_{D_k}$,
using Caccioppoli's inequality as in Lemma \ref{Cacciopoli} and the H\"older regularity estimate as in Lemma \ref{Holder},
we conclude that there exists $\kappa\in(0,1]$ such that
\begin{align}\label{e3.37}
\left(\fint_{\frac{1}{4}B_y}|\delta\nabla u_{3,k}|^2\,dw\right)^{\frac{1}{2}}
&\lesssim\underset{\frac{1}{2}B_y}{\mathrm{osc}}\,u_{3,k}
\lesssim\left(\frac{\delta(y)}{|x_k-y|}\right)^{\kappa}\underset{B_{y,k}
\cap\Omega}{\mathrm{osc}}u_{3,k}\notag\\
&\lesssim\left(\frac{\delta(y)}{|x_k-y|}
\right)^{\kappa}\sup_{w\in B_{y,k}\cap\Omega}|u_{3,k}(w)|,
\end{align}
where $B_{y,k}\supset B_y$ is the largest ball centered at $y$ such that $B_{y,k}\cap4B_k=\emptyset$.
A simple computation yields that, for each $k\in I$, $10B_k\cap B_y=\emptyset$, and hence $B_{y,k}$ is well-defined.
Therefore, from Carleson's inequality \eqref{e2.1}, the fact that $({\rm Loc}_{p_1})_L$ holds, and Lemma \ref{Green}(v), we deduce that
\begin{align}\label{e3.38}
\sup_{w\in B_{y,k}\cap\Omega}|u_{3,k}(w)|
&\lesssim\sup_{w\in B_{y,k}\cap\Omega}\int_{D_k}|\boldsymbol{F}(x)\cdot\nabla_{x}
G_R(x,w)|\,dx\nonumber\\ \nonumber
&\lesssim\sup_{w\in B_{y,k}\cap\Omega}\left\|\widetilde{\mathcal{C}}_1^{\frac{5}{8}}
(\delta\boldsymbol{F}\mathbf{1}_{B_k})\right\|_{L^{p_1'}(\partial\Omega,\sigma)}
\left\|\widetilde{\mathcal{N}}\left(\nabla G_R(\cdot,w)\mathbf{1}_{B_k}\right)\right\|_{L^{p_1}(\partial\Omega,\sigma)}\\ \nonumber
&\lesssim\sup_{w\in B_{y,k}\cap\Omega}\left\|\widetilde{\mathcal{C}}_1^{\frac{5}{8}}\left(\delta\boldsymbol{F}
\mathbf{1}_{B_k}\right)\right\|_{L^{p_1'}(\partial\Omega,\sigma)}
|x_k-y|^{\frac{n-1}{p_1}}\\ \notag &\quad\times\left\{\fint_{2B_k\cap\Omega}[|\nabla_{x}G_R(x,w)|+|G_R(x,w)|]\,dx\right\}\\
&\lesssim\left\|\widetilde{\mathcal{C}}_1^{\frac{5}{8}}\left(\delta\boldsymbol{F}
\mathbf{1}_{B_k}\right)\right\|_{L^{p_1'}(\partial\Omega,\sigma)}|x_k-y|^{-\frac{n-1}{p_1'}}.
\end{align}
Moreover, it was proved in \cite[p.\,35]{fl24} that
$$
\left\|\widetilde{\mathcal{C}}_1^{\frac{5}{8}}\left(\delta\boldsymbol{F}
\mathbf{1}_{B_k}\right)\right\|_{L^{p_1'}(\partial\Omega,\sigma)}
\lesssim\left(\int_{2B_k\cap\partial\Omega}
\left[\widetilde{\mathcal{C}}_1^{\frac{5}{8}}(\delta\boldsymbol{F})
\right]^{p_1'}\,d\sigma\right)^{\frac{1}{p_1'}},
$$
which, together with \eqref{e3.38}, implies that
\begin{align*}
\sup_{w\in B_{y,k}\cap\Omega}|u_{3,k}(w)|&\lesssim\left(\int_{2B_k\cap\partial\Omega}
\left[\widetilde{\mathcal{C}}_1^{\frac{5}{8}}
(\delta\boldsymbol{F})\right]^{p_1'}\,d\sigma\right)^{\frac{1}{p_1'}}
|x_k-y|^{-\frac{n-1}{p_1'}}\\ \nonumber
&\lesssim\left(\fint_{B(z,10|x_k-y|)\cap\partial\Omega}\left\{
\widetilde{\mathcal{C}}_1^{\frac{5}{8}}(\delta\boldsymbol{F})\right\}^{p_1'}
\,d\sigma\right)^{\frac{1}{p_1'}}\\ \nonumber
&\lesssim\left\{\mathcal{M}\left(\left[\widetilde{\mathcal{C}}_1^{\frac{5}{8}}
(\delta\boldsymbol{F})\right]^{p_1'}\right)(z)\right\}^{\frac{1}{p_1'}}.
\end{align*}
This, combined with \eqref{e3.35}, \eqref{e3.36}, and \eqref{e3.37}, further yields that
\begin{align*}
\left(\fint_{\frac{1}{4}B_y}|\delta\nabla u_3|^2\,dw\right)^{\frac{1}{2}}
&\lesssim\sum_{k\in I}\left(\frac{\delta(y)}{|x_k-y|}\right)^{\kappa}\left\{\mathcal{M}\left(
\left[\widetilde{\mathcal{C}}_1^{\frac{5}{8}}(\delta\boldsymbol{F})\right]^{p_1'}
\right)(z)\right\}^{\frac{1}{p_1'}}\\
&\lesssim\left\{\mathcal{M}\left(\left[\widetilde{\mathcal{C}}_1^{\frac{5}{8}}
(\delta\boldsymbol{F})\right]^{p_1'}\right)(z)\right\}^{\frac{1}{p_1'}},
\end{align*}
which, together with \eqref{e3.33} and \eqref{e3.34}, implies that the desired estimate \eqref{e3.31} holds.
This finishes the proof of Lemma \ref{LocwPR}.
\end{proof}

Now, Theorem \ref{StrongLocWeak} is a simple consequence of
Lemmas \ref{strongerLoc}, \ref{selfimprove}, and \ref{LocwPR}.

\begin{proof}[Proof of Theorem \ref{StrongLocWeak}]
By Lemmas \ref{strongerLoc}, \ref{selfimprove}, and \ref{LocwPR},
we conclude that all the conclusions of Theorem \ref{StrongLocWeak} hold.
\end{proof}

\section{Proof of Theorem \ref{RandPR} \label{s4}}

In this section, we give the proof of Theorem \ref{RandPR}.
We begin with proving the equivalence between the solvability of the Robin problems $({\rm R}_p)_L$ and $({\rm \widetilde{R}}_p)_L$.

\begin{proposition}\label{RobinEquiv}
Let $n\ge2$, $p\in(1,\infty)$, $\Omega\subset\mathbb{R}^n$ be a bounded one-sided ${\rm CAD}$, $L:=-\mathrm{div\,}(A\nabla\cdot)$
be a uniformly elliptic operator,
and $\alpha$ be the same as in \eqref{e1.3}. Then the solvability of $({\rm R}_p)_L$ and $({\rm \widetilde{R}}_p)_L$ are equivalent.
\end{proposition}

\begin{proof}
Obviously, the solvability of $({\rm \widetilde{R}}_{p})_{L}$ implies the solvability of $({\rm R}_p)_{L}$.

Now, we prove that, if $({\rm R}_p)_{L}$ is solvable, then $({\rm \widetilde{R}}_{p})_{L}$ is also solvable.
Let $f\in C_{\rm c}(\partial\Omega)$ and $u$ be the weak solution to the Robin problem
\begin{equation*}
\begin{cases}
Lu=0\ \ & \text{in}\ \ \Omega,\\
\displaystyle\frac{\partial u}{\partial\boldsymbol{\nu}}+\alpha  u=f \ \ & \text{on}\ \ \partial\Omega.
\end{cases}
\end{equation*}
We first claim that there exist positive constants $\varepsilon_0$, $a_1$, and $a_2$, depending only on the corkscrew constant of $\Omega$,
such that, for any $z\in\partial\Omega$,
\begin{equation}\label{e4.1}
\mathcal{N}^{a_1}(u)(z)\lesssim\sup_{x\in\Omega_{\varepsilon_0}}|u(x)|+\widetilde{\mathcal{N}}^{a_2}(\nabla u)(z),
\end{equation}
where $\mathcal{N}^{a_1}$ denotes the non-tangential maximal function associated with cone of aperture $a_1$, $\widetilde{\mathcal{N}}^{a_2}$ denotes the
modified non-tangential maximal function associated with cone of aperture $a_2$ (see Definition \ref{NAC}),
and $\Omega_{\varepsilon_0}:=\{z\in\Omega:\mathrm{dist\,}(z,\partial\Omega)\ge\varepsilon_0\}$.
Indeed, fix $r_0\in(0,\mathrm{diam\,}(\Omega))$. Since $\Omega$ satisfies the $C_1$-interior corkscrew condition with some $C_1\in(0,1/2)$,
it follows that, for any given $z\in\partial\Omega$,
there exists $y_z\in\Omega\cap B(z,r_0)$ such that $B(y_z,C_1r_0)\subset B(z,r_0)\cap\Omega$, which implies that $\delta(y_z)\ge C_1r_0$ and hence
$$
|y_z-z|<r_0-C_1r_0\le\frac{1-C_1}{C_1}\delta(y_z).
$$
Take $\varepsilon_0:=C_1r_0$ and $a_1:=(1-C_1)/C_1$. Then $y_z\in\gamma_{a_1}(z)\cap\Omega_{\varepsilon_0}$.
Furthermore, for any $y\in\gamma_{a_1}(z)\setminus\Omega_{\varepsilon_0}$, we have
$$|y-z|<a_1\delta(y)\le\frac{1-C_1}{C_1}\varepsilon_0<r_0,$$
which implies that $y\in B(z,r_0)$.
Moreover, similar to the proof of  \cite[(5.32)]{hs25},
we find that there exists a positive constant $a_2\in(a_1,\infty)$, depending on the corkscrew constant of $\Omega$, such that,
for any $z\in\partial\Omega$ and $x_1,x_2\in\gamma_{a_2}(z)\cap B(z,r_0)$,
$$
|u(x_1)-u(x_2)|\lesssim r_0\widetilde{\mathcal{N}}^{a_2}(\nabla u)(z),
$$
which further implies that, for any $y\in\gamma_{a_1}(z)\setminus\Omega_{\varepsilon_0}$,
$$
|u(y)|\le|u(y)-u(y_z)|+|u(y_z)|\lesssim r_0\widetilde{\mathcal{N}}^{a_2}(\nabla u)(z)+\sup_{x\in\Omega_{\varepsilon_0}}|u(x)|.
$$
Thus, \eqref{e4.1} holds.

Furthermore, from \eqref{e2.10} and \eqref{e2.12}, we deduce that,
for any $x\in\Omega_{\varepsilon_0}$,
$$|u(x)|\le\int_{\partial\Omega}G_R(y,x)|f(y)|\,d\sigma(y)\lesssim
\|f\|_{L^1(\partial\Omega,\sigma)}\lesssim\|f\|_{L^p(\partial\Omega,\sigma)},
$$
which, together with \eqref{RobinEquiv}, implies that
\begin{equation}\label{e4.2}
\left\|\mathcal{N}^{a_1}(u)\right\|_{L^p(\partial\Omega,\sigma)}\lesssim
\left\|\widetilde{\mathcal{N}}^{a_2}(\nabla u)\right\|_{L^p(\partial\Omega,\sigma)}+\|f\|_{L^p(\partial\Omega,\sigma)}.
\end{equation}
Moreover, notice that, for any give $a\in(1,\infty)$ and any $x\in\partial\Omega$, $\widetilde{\mathcal{N}}^{a}(u)(x)\le\mathcal{N}^{a}(u)(x)$,
which, combined with \eqref{e4.2} and Lemma \ref{TentEqui}, further yields
$$
\left\|\widetilde{\mathcal{N}}(u)\right\|_{L^p(\partial\Omega,\sigma)}
\lesssim\left\|\widetilde{\mathcal{N}}(\nabla u)\right\|_{L^p(\partial\Omega,\sigma)}+\|f\|_{L^p(\partial\Omega,\sigma)}.
$$
By this and the solvability of $({\rm R}_p)_L$, we conclude that
$$
\left\|\widetilde{\mathcal{N}}(u)\right\|_{L^p(\partial\Omega,\sigma)}
+\left\|\widetilde{\mathcal{N}}(\nabla u)\right\|_{L^p(\partial\Omega,\sigma)}\lesssim\left\|\widetilde{\mathcal{N}}(\nabla u)\right\|_{L^p(\partial\Omega,\sigma)}+\|f\|_{L^p(\partial\Omega,\sigma)}
\lesssim\|f\|_{L^p(\partial\Omega,\sigma)}.
$$
Thus, $({\rm \widetilde{R}}_p)_L$ is solvable.
This finishes the proof of Proposition \ref{RobinEquiv}.
\end{proof}

To prove Theorem \ref{RandPR}, we divide the conclusion of Theorem \ref{RandPR} into the following four successive lemmas.

\begin{lemma}\label{PRRtoR}
Let $n\ge2$, $\Omega\subset\mathbb{R}^n$ be a bounded one-sided ${\rm CAD}$, $L:=-\mathrm{div\,}(A\nabla\cdot)$ be a uniformly elliptic operator,
$\alpha$ be the same as in \eqref{e1.3}, and $p\in(1,\infty)$. If $({\rm PRR}_{p})_{L}$ is solvable, then $({\rm R}_{p})_{L}$ is solvable.
\end{lemma}

\begin{proof}
Let $p'\in(1,\infty)$ be the H\"older conjugate of $p$.
By Theorem \ref{PRandPRR}, we find that the solvability of $({\rm PRR}_{p})_{L}$ is equivalent to the solvability of $({\rm PR}_{p'})_{L^*}$.
Thus, to prove the conclusion of Lemma \ref{PRRtoR}, it suffices to show that, if
$({\rm PR}_{p'})_{L^*}$ is solvable, then $({\rm R}_{p})_{L}$ is solvable.

Assume that $({\rm PR}_{p'})_{L^*}$ is solvable. Let $f\in L^{p}(\partial\Omega)$ and $u$ be a weak solution to the Robin problem
\begin{equation*}
\begin{cases}
Lu=0 \ \ & \text{in}\ \ \Omega,\\
\displaystyle\frac{\partial u}{\partial\boldsymbol{\nu}}+\alpha  u=f \ \ & \text{on}\ \ \partial\Omega.
\end{cases}
\end{equation*}
We only need to prove
\begin{equation}\label{e4.3}
\left\|\widetilde{\mathcal{N}}(\nabla u)\right\|_{L^{p}(\partial\Omega,\sigma)}\lesssim\|f\|_{L^{p}(\partial\Omega,\sigma)}.
\end{equation}
Let $E$ be a compact set in $\Omega$.
From \eqref{e2.2} and Lemma \ref{Dense}, it follows that
there exists $\boldsymbol{F}:=\boldsymbol{F}_E\in L_{\rm c}^{\infty}(\Omega;\mathbb{R}^n)$
such that
\begin{equation}\label{e4.4}
\left\|\widetilde{\mathcal{C}}_1\left(\delta\boldsymbol{F}\right)
\right\|_{L^{p'}(\partial\Omega,\sigma)}\lesssim1
\end{equation}
and
\begin{equation}\label{e4.5}
\left\|\widetilde{\mathcal{N}}\left(\nabla u\mathbf{1}_E\right)\right\|_{L^{p}(\partial \Omega,\sigma)}
\lesssim \int_{\Omega}\nabla u\cdot\boldsymbol{F}\,dx.
\end{equation}
Let $v$ be the weak solution to the Robin problem
\begin{equation*}
\begin{cases}
L^*v=-\mathrm{div\,}\boldsymbol{F}\ \ & \text{in}\ \ \Omega,\\
\displaystyle\frac{\partial v}{\partial\boldsymbol{\nu}}+\alpha  v=\boldsymbol{F}\cdot\boldsymbol{\nu} \ \ & \text{on}\ \ \partial\Omega.
\end{cases}
\end{equation*}
Then we have
\begin{align*}
\int_{\Omega}\nabla u\cdot\boldsymbol{F}\,dx
&=\int_{\Omega}A^T\nabla v\cdot\nabla u\,dx+\int_{\partial\Omega}\alpha  vu\,d\sigma\\ \nonumber
&=\int_{\Omega}A\nabla u\cdot\nabla v\,dx+\int_{\partial\Omega}\alpha  vu\,d\sigma=\int_{\partial\Omega}fv\,d\sigma.
\end{align*}
By this, H\"older's inequality, the solvability of $({\rm PR}_{p'})_{L^*}$,
\eqref{e4.4}, and \eqref{e4.5}, we conclude that
\begin{align}\label{e4.6}
\left\|\widetilde{\mathcal{N}}(\nabla u\mathbf{1}_E)\right\|_{L^{p}(\partial \Omega,\sigma)}
&\lesssim\left\|f\right\|_{L^{p}(\partial\Omega,\sigma)}\left\|v\right\|_{L^{p'}(\partial\Omega,\sigma)}
\lesssim\left\|f\right\|_{L^{p}(\partial\Omega,\sigma)}
\left\|\widetilde{\mathcal{N}}(v)\right\|_{L^{p'}(\partial\Omega,\sigma)}\nonumber\\
&\lesssim\left\|f\right\|_{L^{p}(\partial\Omega,\sigma)}
\left\|\widetilde{\mathcal{C}}_1(\delta\boldsymbol{F})\right\|_{L^{p'}(\partial\Omega,\sigma)}
\lesssim\left\|f\right\|_{L^{p}(\partial\Omega,\sigma)}.
\end{align}
Letting $E\uparrow\Omega$ in \eqref{e4.6}, we then find that \eqref{e4.3} holds.
This finishes the proof of Lemma \ref{PRRtoR}.
\end{proof}

\begin{lemma}\label{RtoPRR}
Let $n\ge2$, $\Omega\subset\mathbb{R}^n$ be a bounded one-sided {\rm CAD}, $p\in(1,\infty)$, $p'$ be the H\"older conjugate of $p$,
$L:=-\mathrm{div}(A\nabla\cdot)$ be a uniformly elliptic operator, and $\alpha$ be the same as in \eqref{e1.3}.
If $({\rm R}_{p})_{L}$ and $({\rm D}_{p'})_{L^*}$ are solvable, then $({\rm PRR}_{p})_{L}$ is solvable.
\end{lemma}

\begin{proof}
Let $h\in L_{\rm c}^{\infty}(\Omega)$ and $u$ be the weak solution to the Robin problem
\begin{equation}\label{e4.7}
\begin{cases}
Lu=h\ \ & \text{in}\ \ \Omega,\\
\displaystyle\frac{\partial u}{\partial\boldsymbol{\nu}}+\alpha  u=0 \ \ & \text{on}\ \ \partial\Omega.
\end{cases}
\end{equation}
Based on Theorem \ref{PRandPRR}, it suffices to prove
\begin{equation}\label{e4.8}
\left\|\widetilde{\mathcal{N}}(\nabla u)\right\|_{L^{p}(\partial\Omega,\sigma)}\lesssim
\left\|\widetilde{\mathcal{C}}_1(\delta h)\right\|_{L^{p}(\partial\Omega,\sigma)}.
\end{equation}

Let $u_D\in W^{1,2}_0(\Omega)$ be the weak solution to the Dirichlet problem
\begin{equation}\label{e4.9}
\begin{cases}
Lu_D=h\ \ & \text{in}\ \ \Omega,\\
u_D=0 \ \ & \text{on}\ \ \partial\Omega.
\end{cases}
\end{equation}
We claim that there exists some $g\in L^p(\partial\Omega,\sigma)$ such that, for any $x\in\Omega$,
\begin{equation}\label{e4.10}
u(x)-u_D(x)=\int_{\partial\Omega}G_R(y,x)g(y)\,d\sigma(y),
\end{equation}
where $G_R$ is Green's function associated with the Robin problem.
Indeed, define the linear functional $\ell$ by setting, for any $\phi\in {\rm Lip}(\partial\Omega)$,
$$
\ell(\phi):=\int_{\Omega}A\nabla u_D\cdot\nabla\Phi\,dx-\int_{\Omega}h\Phi\,dx,
$$
where $\Phi$ is an extension of $\phi$ in $W^{1,2}(\Omega)\cap {\rm Lip}(\Omega)$.
Here,
$${\rm Lip}(\partial\Omega):=\left\{\phi\in L^\infty(\partial\Omega):[\phi]_{{\rm Lip}(\partial\Omega)}
:=\sup_{\genfrac{}{}{0pt}{}{x,y\in\partial\Omega}{x\neq y}}\frac{|\phi(x)-\phi(y)|}{|x-y|}<\infty\right\}$$
and the definition of ${\rm Lip}(\Omega)$ is similar.
We point out that the functional $\ell$ is well-defined.
Indeed, for any two extensions $\Phi_1,\Phi_2$ of $\phi$, since $u_D$ is a weak solution of \eqref{e4.9} and
$\Phi_1-\Phi_2\in W_0^{1,2}(\Omega)$, it follows that
$$
\int_{\Omega}A\nabla u_D\cdot\nabla(\Phi_1-\Phi_2)\,dx-\int_{\Omega}h(\Phi_1-\Phi_2)\,dx=0.
$$
Therefore, we can take $\Phi$ as the Varopoulos extension of $\phi$ constructed in \cite[Theorem 1.4]{mz23}
(see also \cite[Theorem 1.1]{hr18}).
By \cite[Theorem 1.4]{mz23}, we find that there exists a positive constant $C$ independent of $\phi$ such that
\begin{equation}\label{e4.11}
\left\|\widetilde{\mathcal{N}}(\Phi)\right\|_{L^{p'}(\partial\Omega,\sigma)}+
\left\|\widetilde{\mathcal{C}}_1(\delta\nabla\Phi)\right\|_{L^{p'}(\partial\Omega,\sigma)}
\le C\|\phi\|_{L^{p'}(\partial\Omega,\sigma)}.
\end{equation}
Moreover, from the assumption that $({\rm D}_{p'})_{L^*}$ is solvable and the equivalence of the solvability of $({\rm D}_{p'})_{L^*}$
and the $L^p$ Poisson--Dirichlet-regularity problem proved in \cite[Theorem 1.22]{mpt22} (see also \cite[Theorem 1.3]{fl24}), we infer that
\begin{equation}\label{e4.12}
\left\|\widetilde{\mathcal{N}}(\nabla u_D)\right\|_{L^{p}(\partial\Omega,\sigma)}\lesssim
\left\|\widetilde{\mathcal{C}}_1(\delta h)\right\|_{L^{p}(\partial\Omega,\sigma)},
\end{equation}
which, together with \eqref{e2.1} and \eqref{e4.11}, further implies that
\begin{align*}
|\ell(\phi)|
&\lesssim\left\|\widetilde{\mathcal{N}}(\nabla u_D)\right\|_{L^{p}(\partial\Omega,\sigma)}
\left\|\widetilde{\mathcal{C}}_1(\delta\nabla\Phi)\right\|_{L^{p'}(\partial\Omega,\sigma)}
+\left\|\widetilde{\mathcal{C}}_1(\delta h)\right\|_{L^{p}(\partial\Omega,\sigma)}
\left\|\widetilde{\mathcal{N}}(\Phi)\right\|_{L^{p'}(\partial\Omega,\sigma)}\\\nonumber
&\lesssim\left\|\widetilde{\mathcal{C}}_1(\delta h)\right\|_{L^{p}(\partial\Omega,\sigma)}
\|\phi\|_{L^{p'}(\partial\Omega,\sigma)}.
\end{align*}
Thus, $\ell$ is a bounded functional on $L^{p'}(\partial\Omega,\sigma)$ and there exists $g_D\in L^{p}(\partial\Omega,\sigma)$ such that
$\ell(\phi)=\int_{\partial\Omega}g_D\phi\,d\sigma$ and
\begin{equation}\label{e4.13}
\|g_D\|_{L^{p}(\partial\Omega,\sigma)}\lesssim\left\|\widetilde{\mathcal{C}}_1(\delta h)\right\|_{L^{p}(\partial\Omega,\sigma)}.
\end{equation}
For any $x\in\Omega$, define
$$
v(x):=\int_{\partial\Omega}G_R(y,x)g_D(y)\,d\sigma(y).
$$
Then $v$ is a weak solution to the Robin problem
\begin{equation*}
\begin{cases}
Lv=0\ \ & \text{in}\ \ \Omega,\\
\displaystyle\frac{\partial v}{\partial\boldsymbol{\nu}}+\alpha  v=g_D \ \ & \text{on}\ \ \partial\Omega.
\end{cases}
\end{equation*}
Therefore, we have
\begin{align*}
\int_{\Omega}A\nabla v\cdot\nabla\Phi\,dx+\int_{\partial\Omega}\alpha  v\phi\,d\sigma=\int_{\partial\Omega}g_D\phi\,d\sigma.
\end{align*}
Moreover, since $u$ is a weak solution to the Robin problem \eqref{e4.7}, it follows that, for any $\Phi\in W^{1,2}(\Omega)$
with $\Phi|_{\partial\Omega}=\phi$,
\begin{align*}
&\int_{\Omega}A\nabla(u_D-u)\cdot\nabla\Phi\,dx\\ \notag
&\quad=\ell(\phi)-\left[\int_{\Omega}A\nabla u\cdot\nabla\Phi\,dx+\int_{\partial\Omega}\alpha  u\phi\,d\sigma-\int_{\Omega}h\Phi\,dx\right]
+\int_{\partial\Omega}\alpha  u\phi\,d\sigma\\ \nonumber
&\quad=\int_{\partial\Omega}g_D\phi\,d\sigma+\int_{\partial\Omega}\alpha (u_D-u)\phi\,d\sigma.
\end{align*}
Let $w:=u_D-u-v$. Then, for any $\Phi\in W^{1,2}(\Omega)$ with $\Phi|_{\partial\Omega}=\phi$,
$$
\int_{\Omega}A\nabla w\cdot\nabla\Phi\,dx+\int_{\partial\Omega}\alpha  w\phi\,d\sigma=0.
$$
Taking $\Phi:=w$, we then have $u_D-u-v=0$, which implies that \eqref{e4.10} holds.

Finally, by the solvability of both $({\rm D}_{p'})_{L^*}$ and $({\rm R}_p)_L$, \eqref{e4.12}, and \eqref{e4.13}, we conclude that
\begin{align*}
\left\|\widetilde{\mathcal{N}}(\nabla u)\right\|_{L^p(\partial\Omega,\sigma)}
&\lesssim\left\|\widetilde{\mathcal{N}}(\nabla u_D)\right\|_{L^p(\partial\Omega,\sigma)}+
\left\|\widetilde{\mathcal{N}}(\nabla v)\right\|_{L^p(\partial\Omega,\sigma)}\\ \nonumber
&\lesssim\left\|\widetilde{\mathcal{C}}_1(\delta h)\right\|_{L^p(\partial\Omega,\sigma)}+\|g_D\|_{L^{p}(\partial\Omega,\sigma)}
\lesssim\left\|\widetilde{\mathcal{C}}_1(\delta h)\right\|_{L^p(\partial\Omega,\sigma)},
\end{align*}
which implies that \eqref{e4.8} holds. This finishes the proof of Lemma \ref{RtoPRR}.
\end{proof}

Next, we prove a more general result than Theorem \ref{RandPR}(iii).

\begin{lemma}\label{Rselfimprove}
Let $n\ge2$, $\Omega\subset\mathbb{R}^n$ be a bounded one-sided {\rm CAD}, $1<q<p<r<\infty$,
$L:=-\mathrm{div}(A\nabla\cdot)$ be a uniformly elliptic operator, and $\alpha$ be the same as in \eqref{e1.3}.
If $({\rm R}_q)_L$ is solvable and $({\rm LocR}_r)_L$ holds, then  $({\rm R}_{p})_L$ is solvable.
\end{lemma}

\begin{proof}
By borrowing some ideas from Shen \cite{s07}, we prove the present lemma via using a good-$\lambda$ argument.

Let $f\in C_{\rm c}(\partial\Omega)$ and $u$ be the weak solution to the Robin problem
\begin{equation}\label{e4.14}
\begin{cases}
Lu=0 \ \ & \text{in}\ \ \Omega,\\
\displaystyle\frac{\partial u}{\partial\boldsymbol{\nu}}+\alpha  u=f \ \ & \text{on}\ \ \partial\Omega.
\end{cases}
\end{equation}
It suffices to show
\begin{equation}\label{e4.15}
\left\|\widetilde{\mathcal{N}}(\nabla u)\right\|_{L^{p}(\partial\Omega,\sigma)}\lesssim\|f\|_{L^{p}(\partial\Omega,\sigma)}.
\end{equation}
For any $\lambda\in(0,\infty)$, define
\begin{align*}
E_1(\lambda):=\left\{x\in\partial\Omega:\mathcal{M}\left(\left[\widetilde{\mathcal{N}}(\nabla u)\right]^q\right)(x)>\lambda\right\},\end{align*}
\begin{align*}
E_2(\lambda):=\left\{x\in\partial\Omega:\mathcal{M}\left(
\left[\widetilde{\mathcal{N}}(u)\right]^q\right)(x)>\lambda\right\},\\\nonumber
\end{align*}
and
$$E(\lambda):=\left\{x\in\partial\Omega:\max\left\{\mathcal{M}\left(\left[\widetilde{\mathcal{N}}(\nabla u)\right]^q\right)(x),\mathcal{M}\left(\left[\widetilde{\mathcal{N}}(u)
\right]^q\right)(x)\right\}>\lambda\right\}.
$$
By the equivalence of the solvability of $({\rm R}_q)_L$ and $({\rm \widetilde{R}}_q)_L$ (see Proposition \ref{RobinEquiv})
and the assumption that $({\rm R}_q)_L$ is solvable, we conclude that $\widetilde{\mathcal{N}}(\nabla u),\widetilde{\mathcal{N}}(u)\in L^q(\partial\Omega,\sigma)$,
which, combined with the fact that the Hardy--Littlewood maximal operator $\mathcal{M}$ is weak type $(1,1)$ (see Lemma \ref{l2.1}), implies that, for any $\lambda\in(0,\infty)$,
$\sigma(E(\lambda))<\infty$.

Without loss of generality, for the sake of convenience, we may assume
that $E(\lambda)\neq\partial\Omega$ for any $\lambda\in(0,\infty)$.
We first claim that there exists a positive constants $C$ such that, for any given $\gamma,\eta\in(0,1)$ and $\lambda\in(0,\infty)$,
\begin{equation}\label{e4.16}
\sigma\left(E(\lambda)\cap\left\{x\in\partial\Omega:\mathcal{M}\left(|f|^q\right)(x)
\le\gamma\lambda\right\}\right)\le C\left(\gamma+\eta^{\frac{r}{q}}\right)\sigma(E(\eta\lambda)).
\end{equation}
If \eqref{e4.16} holds, then we have, for any $\lambda\in(0,\infty)$,
$$
\sigma(E(\lambda))\le C\left(\gamma+\eta^{\frac{r}{q}}\right)\sigma(E(\eta\lambda))+\sigma\left(\left\{x\in
\partial\Omega:\mathcal{M}\left(|f|^q\right)(x)>\gamma\lambda\right\}\right),
$$
which, together with the $L^{\frac{p}{q}}$-boundedness of the Hardy--Littlewood maximal operator $\mathcal{M}$ (see Lemma \ref{l2.1}), implies that, for any $\Lambda\in(0,\infty)$,
\begin{align}\label{e4.17}
\int_0^{\Lambda}\sigma(E(\lambda))\lambda^{\frac{p}{q}-1}\,d\lambda
&\le C\left(\gamma+\eta^{\frac{r}{q}}\right)\int_0^{\Lambda}\sigma(E(\eta\lambda))
\lambda^{\frac{p}{q}-1}\,d\lambda\nonumber\\
&\quad+\int_0^{\Lambda}\sigma\left(\left\{x\in\partial\Omega:\mathcal{M}\left(|f|^q\right)(x)>
\gamma\lambda\right\}\right)\lambda^{\frac{p}{q}-1}\,d\lambda\nonumber\\
&\le C\left(\gamma+\eta^{\frac{r}{q}}\right)\eta^{-\frac{p}{q}}
\int_0^{\eta\Lambda}\sigma(E(\lambda))\lambda^{\frac{p}{q}-1}\,d\lambda
+C_{(\gamma)}\int_{\partial\Omega}|f|^p\,d\sigma.
\end{align}
Moreover, from the fact that
$$
\int_0^{\Lambda}\sigma(E(\lambda))\lambda^{\frac{p}{q}-1}\,d\lambda\le
\Lambda^{\frac{p}{q}-1}\sigma(\partial\Omega)<\infty,
$$
we deduce that both the left-hand and the right-hand sides of inequality \eqref{e4.17} are finite. Notice that $r>p>q$.
Take $\eta\in(0,1)$ small enough in \eqref{e4.17} such that $C\eta^{r/q-p/q}\le1/4$. Then, for such an $\eta$,
choose $\gamma\in(0,1)$ small enough in \eqref{e4.17} such that $C\gamma\eta^{-p/q}\le1/4$. For such $\eta$ and $\gamma$, we have
$$C\left(\gamma+\eta^{\frac{r}{q}}\right)\eta^{-\frac{p}{q}}\le\frac{1}{2}.$$
By this and \eqref{e4.17}, we find that
$$
\int_0^{\Lambda}\sigma(E(\lambda))\lambda^{\frac{p}{q}-1}\,d\lambda\le C_{(\gamma)}\int_{\partial\Omega}|f|^p\,d\sigma.
$$
Letting $\Lambda\rightarrow\infty$ in the above inequality, we conclude that
\begin{align}\label{e4.18}
\int_0^{\infty}\sigma(E(\lambda))\lambda^{\frac{p}{q}-1}\,d\lambda\le C_{(\gamma)}\int_{\partial\Omega}|f|^p\,d\sigma,
\end{align}
which, combined with the fact that $E_1(\lambda)\subset E(\lambda)$, further implies that
\begin{align*}
\left\|\widetilde{\mathcal{N}}(\nabla u)\right\|_{L^{p}(\partial\Omega,\sigma)}^p
&\le\int_{\partial\Omega}\left\{\mathcal{M}\left(\left[\widetilde{\mathcal{N}}(\nabla u)\right]^q\right)\right\}^{\frac{p}{q}}\,d\sigma
=\frac{p}{q}\int_0^{\infty}\sigma(E_1(\lambda))\lambda^{\frac{p}{q}-1}\,d\lambda\\ \nonumber
&\lesssim\frac{p}{q}\int_0^{\infty}\sigma(E(\lambda))\lambda^{\frac{p}{q}-1}\,d\lambda
\lesssim\left\|f\right\|_{L^{p}(\partial\Omega,\sigma)}^p.
\end{align*}
Thus, \eqref{e4.15} holds and hence $({\rm R}_p)_L$ is solvable.

Now, we prove \eqref{e4.16}. For any $x\in E(\eta\lambda)$, let $$\Delta(x,r_x):=B(x,r_x)\cap\Omega$$ with
$r_x:=\frac{1}{50}\mathrm{dist\,}(x,\partial\Omega\setminus E(\eta\lambda))$.
Then the union of boundary balls $\{\Delta(x,r_x)\}_{x\in E(\eta\lambda)}$ covers $E(\eta\lambda)$.
By Vitali's covering lemma (see, for example, \cite{s93}), we find that there exist an index set $I$ and a
subcollection $\{\Delta_k:=\Delta(x_k,r_{x_k})\}_{k\in I}$
such that $\{\frac{1}{5}\Delta_k\}_{k\in I}$ is non-overlapping and $E(\eta\lambda)\subset\bigcup_{k\in I}\Delta_k$.
Thus, to prove \eqref{e4.16}, it suffices to show that, for any $\Delta_k$, with $k\in I$, satisfying that there exists $z\in\Delta_k$ such that
\begin{equation}\label{e4.19}
\left\{x\in\Delta_k:\mathcal{M}\left(|f|^q\right)(x)\le\gamma\lambda\right\}\supset\{z\}\neq\emptyset,
\end{equation}
\begin{equation}\label{e4.20}
\sigma\left(E(\lambda)\cap\Delta_k\right)\lesssim\left(\gamma+\eta^{\frac{r}{q}}\right)\sigma(\Delta_k).
\end{equation}

From the definition of $E(\lambda)$, we infer that $E(\lambda)\subset E(\eta\lambda)$ and, for any $k\in I$, there exists $y\in100\Delta_k$ such that
\begin{equation}\label{e4.21}
\mathcal{M}\left(\left[\widetilde{\mathcal{N}}(\nabla u)\right]^q\right)(y)\le\eta\lambda\ \ \text{and}\ \ \mathcal{M}\left(\left[\widetilde{\mathcal{N}}(u)\right]^q\right)(y)\le\eta\lambda.
\end{equation}
Moreover, for any $x\in\Delta_k\cap E(\lambda)$ with $k\in I$,
\begin{align*}
&\lambda<\mathcal{M}\left(\left[\widetilde{\mathcal{N}}(\nabla u)\right]^q\right)(x)=\mathcal{M}\left(\left[\widetilde{\mathcal{N}}(\nabla u\mathbf{1}_{2B_k})\right]^q\right)(x)
\end{align*}
and
\begin{align*}
&\lambda<\mathcal{M}\left(\left[\widetilde{\mathcal{N}}(u)\right]^q\right)(x)
=\mathcal{M}\left(\left[\widetilde{\mathcal{N}}
(u\mathbf{1}_{2B_k})\right]^q\right)(x),
\end{align*}
where $B_k:=B(x_k,r_{x_k})$. Fix $k\in I$. Let $u_0$ be a weak solution to the Robin problem \eqref{e4.14} with $f$ replaced by $f\mathbf{1}_{4\Delta_k}$.
Then we have
\begin{align}\label{e4.22}
\sigma(E(\lambda)\cap\Delta_k)
&\le\sigma\left(\left\{x\in\Delta_k:\mathcal{M}\left(\left[\widetilde{\mathcal{N}}(\nabla u)\right]^q\right)(x)>\lambda\right\}\right)\nonumber\\ \notag
&\quad+\sigma\left(\left\{x\in\Delta_k:\mathcal{M}\left(\left[\widetilde{\mathcal{N}}
(u)\right]^q\right)(x)>\lambda\right\}\right)\\ \nonumber
&=\sigma\left(\left\{x\in\Delta_k:\mathcal{M}\left(\left[\widetilde{\mathcal{N}}(\nabla u\mathbf{1}_{2B_k})\right]^q\right)(x)>\lambda\right\}\right)\\ \notag
&\quad+\sigma\left(\left\{x\in\Delta_k:\mathcal{M}\left(\left[\widetilde{\mathcal{N}}
(u\mathbf{1}_{2B_k})\right]^q\right)(x)>\lambda\right\}\right)\\ \nonumber
&\le\sigma\left(\left\{x\in\Delta_k:\mathcal{M}\left(\left[\widetilde{\mathcal{N}}(\nabla[u-u_0]
\mathbf{1}_{2B_k})\right]^q\right)(x)>\lambda\right\}\right)\\ \notag
&\quad+\sigma\left(\left\{x\in\Delta_k:\mathcal{M}\left(\left[\widetilde{\mathcal{N}}(\nabla u_0\mathbf{1}_{2B_k})\right]^q\right)(x)>\lambda\right\}\right)\\ \nonumber
&\quad+\sigma\left(\left\{x\in\Delta_k:\mathcal{M}\left(\left[\widetilde{\mathcal{N}}([u-u_0]
\mathbf{1}_{2B_k})\right]^q\right)(x)>\lambda\right\}\right)\\ \notag
&\quad+\sigma\left(\left\{x\in\Delta_k:\mathcal{M}\left(\left[
\widetilde{\mathcal{N}}(u_0\mathbf{1}_{2B_k})\right]^q\right)
(x)>\lambda\right\}\right)\\
&=:\mathrm{I}_1+\mathrm{I}_2+\mathrm{I}_3+\mathrm{I}_4.
\end{align}

For the term $\mathrm{I}_2$, by the fact that the Hardy--Littlewood
maximal operator $\mathcal{M}$ is of weak type $(1,1)$ (see Lemma \ref{l2.1}),
the solvability of $({\rm R}_q)_L$, and \eqref{e4.19}, we conclude that
\begin{align}\label{e4.23}
\mathrm{I}_2
&\lesssim\frac{1}{\lambda}\int_{\partial\Omega}\left[\widetilde{\mathcal{N}}(\nabla u_0\mathbf{1}_{2B_k})\right]^q\,d\sigma
\lesssim\frac{1}{\lambda}\left\|f\mathbf{1}_{4\Delta_k}\right\|_{L^q(\partial\Omega,\sigma)}^q
\nonumber\\
&\lesssim\frac{1}{\lambda}\int_{4\Delta_k}|f|^q\,d\sigma
\lesssim\frac{\sigma(\Delta_k)}{\lambda}\mathcal{M}\left(|f|^q\right)(z)
\lesssim\gamma\sigma(\Delta_k).
\end{align}
For $\mathrm{I}_1$, from the fact that $u-u_0$ is a weak solution to $Lw=0$ in $4B_k\cap\Omega$ with the Robin boundary condition $\frac{\partial w}{\partial\boldsymbol{\nu}}+\alpha  w=0$ on $4\Delta_k$, the assumption that
$({\rm LocR}_r)_L$ holds with $r\in(p,\infty)$,
 and the $L^{\frac{r}{q}}$-boundedness of $\mathcal{M}$ (see Lemma \ref{l2.1}), we deduce that
\begin{align}\label{e4.24}
\mathrm{I}_1
&\lesssim\lambda^{-\frac{r}{q}}\int_{\partial\Omega}\left[\widetilde{\mathcal{N}}(\nabla[u-u_0]
\mathbf{1}_{2B_k})\right]^r\,d\sigma\nonumber\\
&\lesssim\lambda^{-\frac{r}{q}}\sigma(\Delta_k)\left[\fint_{4B_k\cap\Omega}
|\nabla(u-u_0)|\,dx+\fint_{4B_k\cap\Omega}|u-u_0|\,dx\right]^{r}.
\end{align}
Moreover, similar to the estimation of \eqref{e3.26}, we find that
\begin{align*}
\fint_{4B_k\cap\Omega}|\nabla(u-u_0)|\,dx
&\lesssim\fint_{20B_k\cap\partial\Omega}\widetilde{\mathcal{N}}\left(\nabla[u-u_0]
\mathbf{1}_{4B_k}\right)\,d\sigma\\\nonumber
&\lesssim \fint_{20\Delta_k}\widetilde{\mathcal{N}}(\nabla u)\,d\sigma+\fint_{20\Delta_k}\widetilde{\mathcal{N}}(\nabla u_0)\,d\sigma.
\end{align*}
From this, H\"older's inequality, \eqref{e4.21}, the solvability of $({\rm R}_q)_L$, and \eqref{e4.19}, it follows that
\begin{align}\label{e4.25}
\fint_{4B_k\cap\Omega}|\nabla(u-u_0)|\,dx
&\lesssim\left\{\fint_{20\Delta_k}\left[\widetilde{\mathcal{N}}(\nabla u)\right]^q\,d\sigma\right\}^{\frac{1}{q}}
+\left\{\fint_{20\Delta_k}\left[\widetilde{\mathcal{N}}(\nabla u_0)\right]^q\,d\sigma\right\}^{\frac{1}{q}}\nonumber\\\nonumber
&\lesssim\left\{\mathcal{M}\left(\left[\widetilde{\mathcal{N}}(\nabla u)\right]^q\right)(y)\right\}^{\frac{1}{q}}+\sigma(\Delta_k)
^{-\frac{1}{q}}\left\|f\mathbf{1}_{4\Delta_k}\right\|_{L^q(\partial\Omega,\sigma)}\\\nonumber
&\lesssim(\eta\lambda)^{\frac{1}{q}}+\left(\fint_{4\Delta_k}|f|^q\,d\sigma\right)^{\frac{1}{q}}\\
&\lesssim(\eta\lambda)^{\frac{1}{q}}+\left[\mathcal{M}\left(|f|^q\right)(z)\right]^{\frac{1}{q}}
\lesssim\lambda^{\frac{1}{q}}\left(\eta^{\frac{1}{q}}+\gamma^{\frac{1}{q}}\right).
\end{align}
Similarly, by H\"older's inequality, \eqref{e4.21}, the equivalence of the solvability of $({\rm R}_q)_L$ and
$({\rm \widetilde{R}}_q)_L$ (see Proposition \ref{RobinEquiv}), and \eqref{e4.19}, we deduce that
$$
\fint_{4B_k\cap\Omega}|u-u_0|\,dx
\lesssim\lambda^{\frac{1}{q}}\left(\eta^{\frac{1}{q}}+\gamma^{\frac{1}{q}}\right),
$$
which, together with \eqref{e4.24} and \eqref{e4.25}, further implies that
\begin{equation}\label{e4.26}
\mathrm{I}_1\lesssim\left(\eta^{\frac{r}{q}}+\gamma^{\frac{r}{q}}\right)\sigma(\Delta_k)
\lesssim\left(\eta^{\frac{r}{q}}+\gamma\right)\sigma(\Delta_k).
\end{equation}
For the terms $\mathrm{I}_3$ and $\mathrm{I}_4$, using the equivalence of the solvability of $({\rm R}_q)_L$ and
$({\rm \widetilde{R}}_q)_L$ and the fact that \eqref{e3.22} holds for any $r\in(p,\infty)$, similar to
the estimation of \eqref{e4.23} and \eqref{e4.26}, we conclude that
$$
\mathrm{I}_3,\mathrm{I}_4\lesssim\left(\eta^{\frac{r}{q}}+\gamma\right)
\sigma(\Delta_k),
$$
which, combined with \eqref{e4.23}, \eqref{e4.26}, and \eqref{e4.22}, further implies that \eqref{e4.20} and hence \eqref{e4.16} hold.
This finishes the proof of Lemma \ref{Rselfimprove}.
\end{proof}

Finally, we employ the same approach as that used in the proof of Lemma \ref{Rselfimprove} to prove the extrapolation
property of the $L^p$ Poisson--Robin problem.

\begin{lemma}\label{PRselfimprove}
Let $n\ge2$, $\Omega\subset\mathbb{R}^n$ be a bounded one-sided ${\rm CAD}$, $L:=-\mathrm{div\,}(A\nabla\cdot)$ be a
uniformly elliptic operator, $\alpha$ be the same as in \eqref{e1.3}, and $q\in(1,\infty)$.
If $({\rm PR}_q)_L$ is solvable, then, for any $p\in[q,\infty)$, $({\rm PR}_{p})_{L}$ is solvable.
\end{lemma}

\begin{proof}
Let $\boldsymbol{F}\in L_{\rm c}^{\infty}(\Omega;\mathbb{R}^n)$ and $u$ be the weak solution to the Robin problem
\begin{equation*}
\begin{cases}
Lu=-\mathrm{div\,}\boldsymbol{F}\ \ & \text{in}\ \ \Omega,\\
\displaystyle\frac{\partial u}{\partial\boldsymbol{\nu}}+\alpha  u=\boldsymbol{F}\cdot\boldsymbol{\nu} \ \ & \text{on}\ \ \partial\Omega.
\end{cases}
\end{equation*}
To prove the present lemma, it suffices to show that, for any $p\in(q,\infty)$,
\begin{equation}\label{e4.27}
\left\|\widetilde{\mathcal{N}}(u)\right\|_{L^{p}(\partial\Omega,\sigma)}
\lesssim\left\|\widetilde{\mathcal{C}}_1(\delta\boldsymbol{F})\right\|_{L^{p}(\partial\Omega,\sigma)}.
\end{equation}
For any $\lambda\in(0,\infty)$, define
$$
E(\lambda):=\left\{x\in\partial\Omega:\mathcal{M}\left(\left[
\widetilde{\mathcal{N}}(u)\right]^q\right)(x)>\lambda\right\}.
$$
By the assumption that $({\rm PR}_q)_L$ is solvable, we conclude that
$\widetilde{\mathcal{N}}(u)\in L^q(\partial\Omega,\sigma)$, which, together with
the fact that the Hardy--Littlewood maximal operator $\mathcal{M}$ is of weak type $(1,1)$ (see Lemma \ref{l2.1}),
implies that, for any $\lambda\in(0,\infty)$, $\sigma(E(\lambda))<\infty.$ Without loss of generality, for the sake of convenience,
we may assume that $E(\lambda)\neq\partial\Omega$ for any $\lambda\in(0,\infty)$.

Let $p\in(q,\infty)$ and $r\in(p,\infty)$. Similar to the proof of the good-$\lambda$ inequality \eqref{e4.16},
we find that there exists a positive constant
$C$ such that, for any given $\eta,\gamma\in(0,1)$ small enough and for any $\lambda\in(0,\infty)$,
\begin{equation}\label{e4.28}
\sigma\left(E(\lambda)\cap\left\{x\in\partial\Omega:\mathcal{M}
\left(\left[\widetilde{\mathcal{C}}_1(\delta\boldsymbol{F})
\right]^q\right)(x)\le\gamma\lambda\right\}\right)\le C\left(\gamma+\eta^{\frac{r}{q}}\right)\sigma(E(\eta\lambda)).
\end{equation}
Then, from \eqref{e4.28}, we deduce that, for any $\lambda\in(0,\infty)$,
$$
\sigma(E(\lambda))\le C\left(\gamma+\eta^{\frac{r}{q}}\right)\sigma(E(\eta\lambda))+\sigma\left(\left\{x\in\partial\Omega:
\mathcal{M}\left(\left[\widetilde{\mathcal{C}}_1(\delta
\boldsymbol{F})\right]^q\right)(x)>\gamma\lambda\right\}\right),
$$
which, combined with the $L^{\frac{p}{q}}$-boundedness of the Hardy--Littlewood maximal operator $\mathcal{M}$
(see Lemma \ref{l2.1}), implies that, for any given $\Lambda\in(0,\infty)$,
\begin{align}\label{e4.29}
\int_0^{\Lambda}\sigma(E(\lambda))\lambda^{\frac{p}{q}-1}\,d\lambda
&\le C\left(\gamma+\eta^{\frac{p}{q}}\right)\int_0^{\Lambda}\sigma(E(\eta\lambda))
\lambda^{\frac{p}{q}-1}\,d\lambda\nonumber\\\nonumber
&\quad+\int_0^{\Lambda}\sigma\left(\left\{x\in\partial\Omega:
\mathcal{M}\left(\left[\widetilde{\mathcal{C}}_1(\delta\boldsymbol{F})
\right]^q\right)(x)>\gamma\lambda\right\}\right)
\lambda^{\frac{p}{q}-1}\,d\lambda\\
&\le C\left(\gamma+\eta^{\frac{r}{q}}\right)\eta^{-\frac{p}{q}}\int_0^{\eta\Lambda}\sigma(E(\lambda))
\lambda^{\frac{p}{q}-1}\,d\lambda
+C_{(\gamma)}\int_{\partial\Omega}\left[\widetilde{\mathcal{C}}_1(\delta\boldsymbol{F})\right]^p\,d\sigma.
\end{align}
Using \eqref{e4.29}, similar to the proof of \eqref{e4.18},
we conclude that
$$
\int_0^{\infty}\sigma(E(\lambda))\lambda^{\frac{p}{q}-1}\,d\lambda\le C_{(\gamma)}\int_{\partial\Omega}
\left[\widetilde{\mathcal{C}}_1(\delta\boldsymbol{F})\right]^p\,d\sigma,
$$
which further implies that
\begin{align*}
\left\|\widetilde{\mathcal{N}}(u)\right\|_{L^{p}(\partial\Omega,\sigma)}^p
&\le\int_{\partial\Omega}\left\{\mathcal{M}\left(\left[\widetilde{\mathcal{N}}
(u)\right]^q\right)\right\}^{\frac{p}{q}}\,d\sigma
=\frac{p}{q}\int_0^{\infty}\sigma(E(\lambda))\lambda^{\frac{p}{q}-1}\,d\lambda\\
&\lesssim\left\|\widetilde{\mathcal{C}}_1(\delta\boldsymbol{F})\right\|_{L^{p}(\partial\Omega,\sigma)}^p.
\end{align*}
Thus, \eqref{e4.27} holds and hence $({\rm PR}_p)_L$ is solvable.
This finishes the proof of Lemma \ref{PRselfimprove}.
\end{proof}

Using Lemmas \ref{PRRtoR}, \ref{RtoPRR}, \ref{Rselfimprove}, and \ref{PRselfimprove}, we finish the proof of Theorem \ref{RandPR}.

\begin{proof}[Proof of Theorem \ref{RandPR}]
The conclusions of (i) through (iv) of the present theorem follow from Lemmas \ref{PRRtoR}, \ref{RtoPRR},
\ref{Rselfimprove}, and \ref{PRselfimprove}, respectively.
\end{proof}

\section{Applications to Laplace operators on Lipschitz domains\label{s5}}
In this section, we give an application of Theorems \ref{PRandPRR} and \ref{RandPR} to the Poisson--Robin(-regularity)
problem for the Laplace operator on bounded Lipschitz domains.

Let $n\ge3$, $\Omega\subset\mathbb{R}^n$ be a bounded Lipschitz domain, $p\in(1,\infty)$,
and $\alpha$ be the same as in \eqref{e1.3}. In this case, we denote the $L^p$ Dirichlet and the $L^p$ Robin problems of
the Laplace equation $\Delta u=0$ in $\Omega$ simply by $({\rm D}_p)_{\Delta}$ and $({\rm R}_p)_{\Delta}$.
The solvability of the Dirichlet problem $({\rm D}_p)_{\Delta}$ and the Robin problem
$({\rm R}_p)_{\Delta}$ is well known (see, for example, \cite{d79,dk87,k94,ls05,v84}).
\begin{proposition}\label{DNR}
Let $n\ge3$, $\Omega\subset\mathbb{R}^n$ be a bounded Lipschitz domain, and $\alpha$ be the same as in \eqref{e1.3}.
Then the following assertions hold.
\begin{enumerate}[{\rm (i)}]

\item There exists a positive constant $\varepsilon\in(0,1]$, depending only on the Lipschitz constant of $\Omega$ and $n$,
such that, for any $p\in(2-\varepsilon,\infty)$, the Dirichlet problem $({\rm D}_p)_{\Delta}$ is solvable.

\item There exists a positive constant $\varepsilon\in(0,\infty)$, depending only on the Lipschitz constant of $\Omega$ and $n$,
such that, for any $p\in(1,2+\varepsilon)$, the Robin problem $({\rm R}_p)_{\Delta}$ is solvable.
\end{enumerate}
\end{proposition}

As applications of  Theorems \ref{PRandPRR} and \ref{RandPR} and Proposition \ref{DNR}, we now consider the solvability of
the $L^p$ Poisson--Robin problem and the $L^p$ Poisson--Robin-regularity problem when $\Omega\subset\mathbb{R}^n$ is a bounded Lipschitz
domain and $L:=\Delta$. In this case, we denote the $L^p$ Poisson--Robin problem $({\rm PR}_p)_L$ and the $L^p$ Poisson--Robin-regularity problem
$({\rm PRR}_p)_L$ simply by $({\rm PR}_p)_\Delta$ and $({\rm PRR}_p)_\Delta$, respectively.

From Theorems \ref{PRandPRR} and \ref{RandPR} and Proposition \ref{DNR}, we directly deduce the following conclusion
on the solvability of the problems $({\rm PR}_p)_\Delta$ and $({\rm PRR}_p)_\Delta$.

\begin{theorem}\label{DeltaPRR}
Let $n\ge3$, $\Omega\subset\mathbb{R}^n$ be a bounded Lipschitz domain, and $\alpha$ be the same as in \eqref{e1.3}.
Then the following assertions hold.

\begin{enumerate}[{\rm (i)}]
\item There exists a positive constant $\varepsilon\in(0,1]$, depending only on the Lipschitz constant of $\Omega$ and $n$, such that,
for any $p\in(2-\varepsilon,\infty)$, the Poisson--Robin problem $({\rm PR}_p)_{\Delta}$ is solvable.

\item There exists a positive constant $\varepsilon\in(0,\infty)$, depending only on the Lipschitz constant of $\Omega$ and $n$, such that, for any
$p\in(1,2+\varepsilon)$, the Poisson--Robin-regularity problem $({\rm PRR}_p)_{\Delta}$ is solvable.
\end{enumerate}
\end{theorem}

\begin{remark}
We point out that the ranges of $p$ for the solvability of the Poisson--Robin problem $({\rm PR}_p)_{\Delta}$ and
the Poisson--Robin-regularity problem $({\rm PRR}_p)_{\Delta}$ in Theorem \ref{DeltaPRR} are sharp.
Indeed, it is known that the range $p\in(1,2+\varepsilon)$ of $p$ for the solvability of the Robin problem
$({\rm R}_p)_{\Delta}$ is sharp (see, for instance, \cite[Theorems 2.8 and 4.3]{ls05}). From this and Theorems \ref{PRandPRR}(i) and \ref{RandPR}(i), we deduce
that the range $(2-\varepsilon,\infty)$ of $p$ for the solvability of the Poisson--Robin problem $({\rm PR}_p)_{\Delta}$ and
the range $p\in(1,2+\varepsilon)$ of $p$ for the solvability of the Poisson--Robin-regularity problem $({\rm PRR}_p)_{\Delta}$ are sharp.
\end{remark}

Moreover, when $\Omega$ is a bounded $C^1$ domain, we have the following solvability of the Dirichlet problem $({\rm D}_p)_{\Delta}$
and the Robin problem $({\rm R}_p)_{\Delta}$ with any $p\in(1,\infty)$ (see, for example, \cite{fjr78,k94,ls05}).

\begin{proposition}\label{p1}
Let $n\ge3$, $\Omega\subset\mathbb{R}^n$ be a bounded $C^1$ domain, $p\in(1,\infty)$, and $\alpha$ satisfy \eqref{e1.3} and
$\alpha\in L^p(\partial\Omega,\sigma)$ when $p\in(n-1,\infty)$.
Then the Dirichlet problem $({\rm D}_p)_{\Delta}$ and the Robin problem $({\rm R}_p)_{\Delta}$ are solvable.
\end{proposition}

As corollaries of Theorems \ref{PRandPRR} and \ref{RandPR} and Proposition \ref{p1}, we have
the following solvability of $({\rm PR}_p)_\Delta$ and $({\rm PRR}_p)_\Delta$ in the case of bounded $C^1$ domains.

\begin{corollary}\label{DeltaPR}
Let $n\ge3$, $\Omega\subset\mathbb{R}^n$ be a bounded $C^1$ domain, $p\in(1,\infty)$, and $\alpha$ satisfy \eqref{e1.3} and
$\alpha\in L^p(\partial\Omega,\sigma)$ when $p\in(n-1,\infty)$. Then the Poisson--Robin problem
$({\rm PR}_p)_{\Delta}$ and the Poisson--Robin-regularity problem $({\rm PRR}_p)_{\Delta}$ are solvable.
\end{corollary}

\medskip

\noindent\textbf{Data Availability}\quad  Data sharing not applicable
to this article as no datasets were generated or analyzed during
the current study.

%

\bigskip

\noindent Xuelian Fu and Sibei Yang

\medskip

\noindent School of Mathematics and Statistics, Gansu Key Laboratory of Applied Mathematics
and Complex Systems, Lanzhou University, Lanzhou 730000, The People's Republic of China

\smallskip

\smallskip

\noindent {\it E-mails}: \texttt{fuxl2023@lzu.edu.cn} (X. Fu)

\hspace{0.888cm} \texttt{yangsb@lzu.edu.cn} (S. Yang)

\bigskip
	
\noindent Dachun Yang (Corresponding author)
	
\medskip
	
\noindent Laboratory of Mathematics and Complex Systems (Ministry of Education of China),
School of Mathematical Sciences, Beijing Normal University, Beijing 100875,
The People's Republic of China
	
\smallskip
	
\noindent{\it E-mail:} \texttt{dcyang@bnu.edu.cn}
	
\end{document}